\documentclass[11pt]{article}

\usepackage[a4paper,margin=1in]{geometry}
\usepackage[T1]{fontenc}
\usepackage{microtype}
\usepackage{amsmath,amssymb,mathtools,amsthm}
\usepackage{graphicx}
\usepackage{booktabs}
\providecommand{\botrule}{\bottomrule}
\usepackage{float}
\usepackage{tikz}
\usetikzlibrary{plotmarks,positioning}
\usepackage{subfig}
\usepackage{aliascnt}
\usepackage[authoryear,round]{natbib}
\usepackage{xcolor}
\usepackage[colorlinks=true,allcolors=blue]{hyperref}
\usepackage[capitalise,noabbrev]{cleveref}
\providecommand{\doi}[1]{}
\renewcommand{\doi}[1]{\href{https://doi.org/#1}{doi:\,\nolinkurl{#1}}}

\theoremstyle{plain}
\newtheorem{theorem}{Theorem}
\newaliascnt{proposition}{theorem}
\newtheorem{proposition}[proposition]{Proposition}
\aliascntresetthe{proposition}
\newaliascnt{lemma}{theorem}
\newtheorem{lemma}[lemma]{Lemma}
\aliascntresetthe{lemma}
\newaliascnt{corollary}{theorem}
\newtheorem{corollary}[corollary]{Corollary}
\aliascntresetthe{corollary}
\newaliascnt{definition}{theorem}
\newtheorem{definition}[definition]{Definition}
\aliascntresetthe{definition}
\newaliascnt{assumption}{theorem}
\newtheorem{assumption}[assumption]{Assumption}
\aliascntresetthe{assumption}
\theoremstyle{remark}
\newaliascnt{remark}{theorem}
\newtheorem{remark}[remark]{Remark}
\aliascntresetthe{remark}
\numberwithin{equation}{section}
\crefname{theorem}{Theorem}{Theorems}
\crefname{proposition}{Proposition}{Propositions}
\crefname{lemma}{Lemma}{Lemmas}
\crefname{corollary}{Corollary}{Corollaries}
\crefname{definition}{Definition}{Definitions}
\crefname{assumption}{Assumption}{Assumptions}
\crefname{remark}{Remark}{Remarks}

\begin{document}

\title{A penalty-free bulk--surface CutFEM stabilized by lattice Green's function extensions}
\author{Qing Xia\\[0.4em]
\small College of Science, Mathematics and Technology, Wenzhou-Kean University\\
\small Wenzhou, Zhejiang 325060, China\\
\small \href{mailto:qxia@kean.edu}{\texttt{qxia@kean.edu}}}
\date{}
\maketitle

\begin{abstract}
We introduce a penalty-free cut finite element method for surface elliptic problems coupled to a harmonic bulk field on a Cartesian grid. Instead of adding a stabilization term, the method restricts the active finite element space by a discrete bulk harmonic extension represented with the lattice Green's function, together with a local extrapolation near the interface. The resulting method is a symmetric Galerkin scheme posed on a reduced space embedded in the standard active-mesh finite element space. Under stated geometric, regularity, and approximation assumptions, we establish optimal $\mathcal{O}(h)$ and $\mathcal{O}(h^2)$ surface convergence rates, as well as robust, cut-independent algebraic conditioning. Furthermore, a density formulation based on lattice layer potentials is shown to act as an operator preconditioner; the single-layer parametrization yields an $\mathcal{O}(1)$ algebraically well-conditioned system without introducing any tunable parameters. Two-dimensional experiments on circular, deformed, and smooth nonconvex interfaces confirm the predicted convergence and robustness under changes in the cut position. Finally, a three-dimensional torus experiment demonstrates the identical construction using a seven-point bulk stencil and trilinear surface traces, exhibiting the expected optimal error rates.
\end{abstract}



\section{Introduction}\label{sec:intro}

Partial differential equations posed on surfaces, or coupling bulk and surface phenomena, arise in applications ranging from surface diffusion on biological membranes and cell signalling~\citep{dziuk2013finite,elliott2012modelling} to thin-film flow~\citep{roy2002lubrication} and coupled boundary-layer problems in fluid mechanics~\citep{elliott2013finite}. For a smooth closed curve $\Gamma\subset\mathbb{R}^2$, the canonical surface problem is the Laplace--Beltrami equation $-\Delta_\Gamma u=g$ on $\Gamma$. In many applications it is coupled to a bulk equation in a domain $\Omega$ with $\partial\Omega=\Gamma$; in the simplest setting, the surface unknown is the trace of a bulk-harmonic field.

Unfitted discretizations such as CutFEM~\citep{burman2015cutfem} and TraceFEM~\citep{olshanskii2018trace,olshanskii2009finite} are attractive for evolving or implicitly defined geometries, but they suffer from a \emph{small-cut instability}: basis functions with vanishingly small support on $\Gamma$ are poorly controlled by the surface bilinear form, and the condition number of the naive Galerkin assembly can grow rapidly~\citep{de2023stability,de2017condition,olshanskii2018trace}. Stabilization strategies include ghost penalties~\citep{burman2010ghost,burman2015stabilized,burman2016cut}, cell aggregation~\citep{badia2018aggregated,chen2023arbitrarily,hansbo2016cut}, discrete polynomial extension~\citep{burman2022cutfem}, and shifted-boundary methods~\citep{main2018shifted}. Most introduce a parameter, modify the bilinear form, or alter the local discrete space near the interface.

The idea of this paper is to use the bulk equation itself as the stabilizing mechanism. In a coupled problem with a harmonic bulk field, the active-grid degrees of freedom cannot vary independently. We determine the interior active layer by a discrete harmonic extension from an exterior stencil-adjacent layer and determine the remaining exterior active vertices by local extrapolation. On a Cartesian grid, the harmonic extension is represented by lattice Green's function (LGF) layer potentials~\citep{liska2014parallel,liska2016fast,martinsson2009boundary,xia2026geometrically}. The resulting functions form a reduced subspace of the standard active finite element space. Restricting both trial and test functions to this subspace gives a symmetric Galerkin system without a ghost penalty, normal-gradient penalty, aggregation rule, or nonsymmetric projection.

We formulate the coupled bulk--surface problem on this LGF-constrained space and establish optimal surface error estimates under explicit regularity and approximation assumptions. For the direct reduced operator, the conditioning is independent of the cut position and has the usual mesh-size scaling of a second-order surface problem in the flat periodic setting; the curved-interface counterpart is supported by direct numerical measurements. We also introduce single- and double-layer density parametrizations and explain their different spectral behaviour through a flat-interface symbol calculation. The algebraic construction has two complementary roles: the harmonic constraint removes weakly observed small-cut modes, while the single-layer density coordinates reduce the effective order of the system. Numerical experiments confirm the predicted convergence and robustness under translations of circular, deformed, and smooth nonconvex interfaces.

The closest conceptual predecessor is the discrete-extension framework of \citet{burman2022cutfem}: both methods define a reduced trial space inside the active finite element space. The difference is the extension mechanism. In that framework, data are continued polynomially from stable interior elements; here they are continued by the discrete harmonic extension associated with the bulk operator. The construction is also related to difference-potential methods~\citep{ryaben2012method,epshteyn2020difference} and to fast lattice solvers based on discrete potential theory~\citep{gillman2014fast}.

The analysis is developed for smooth closed curves on a Cartesian grid with the five-point Laplacian as the model bulk operator; smooth nonconvex domains are included. Interfaces with corners are not covered by the present regularity theory. The same framework applies to screened Poisson and suitable constant-coefficient anisotropic scalar stencils. The numerical results also include a three-dimensional torus computed with a seven-point bulk stencil and trilinear surface traces. That experiment exhibits second-order bulk and surface $L^2$ convergence and first-order surface $H^1$ convergence; the reduced condition numbers and iteration counts remain of order $10^2$ over the three tested levels.

The remainder of the paper is organized as follows. \Cref{sec:problem} states the coupled problem and fixes notation. \Cref{sec:cutfem,sec:lgf-extension,sec:reduced} introduce the CutFEM discretization, the LGF harmonic extension, and the reduced Galerkin formulations. \Cref{sec:analysis} develops the approximation and surface error analysis. \Cref{sec:conditioning-section} establishes trace stability, bulk reconstruction, and conditioning, and discusses extensions to other bulk operators. \Cref{sec:numerics} presents the two- and three-dimensional numerical experiments, and \Cref{sec:conclusion} concludes.

\section{Problem setting}\label{sec:problem}

Let $\Omega\subset\mathbb{R}^2$ be a bounded domain with smooth boundary $\Gamma=\partial\Omega$ and outward unit normal $\mathbf n$. The domain $\Omega$ may be nonconvex; the only geometric requirement is that $\Gamma$ is a smooth closed curve. We consider the coupled bulk--surface problem
\begin{subequations}\label{eq:coupled}
\begin{align}
    -\Delta u &= 0 \quad\text{in } \Omega, \label{eq:coupled-bulk}\\
    -\Delta_{\Gamma} u &= g \quad\text{on } \Gamma, \label{eq:coupled-surf}
\end{align}
\end{subequations}
with surface data $g$ satisfying the compatibility condition $\int_\Gamma g\,\mathrm{d}s=0$ and gauge normalization $\int_\Gamma u\,\mathrm{d}s=0$. The surface unknown is the trace of the bulk unknown: \eqref{eq:coupled-bulk} determines $u$ in $\Omega$ as the harmonic extension of its boundary values, and \eqref{eq:coupled-surf} determines those boundary values from the surface data $g$.

\begin{remark}[Nature of the coupling]\label{rem:strong-bulk}
At the continuous level, \eqref{eq:coupled-surf} determines $u|_\Gamma$ and \eqref{eq:coupled-bulk} supplies its harmonic extension. Discretely, bulk harmonicity is imposed strongly by eliminating the interior active degrees of freedom, while only the surface equation is posed variationally on the reduced trace space. Thus ``bulk--surface'' denotes a strong trace constraint rather than a two-way exchange law.
\end{remark}

Define the mean-zero subspace
\begin{equation}\label{eq:H1diamond}
    H^1_\diamond(\Gamma) := \Bigl\{ v \in H^1(\Gamma) :
    \int_\Gamma v \, \mathrm{d}s = 0 \Bigr\}.
\end{equation}
The variational form of \eqref{eq:coupled-surf} is: find $u\in H^1_\diamond(\Gamma)$ such that 
\begin{equation}\label{eq:weak-surf} 
    a_\Gamma(u,v) = (g,v)_\Gamma,\qquad \forall v\in H^1_\diamond(\Gamma),
\end{equation}
where $a_\Gamma(u,v):=\int_\Gamma\nabla_\Gamma u\cdot\nabla_\Gamma v\,\mathrm{d}s$ and $(g,v)_\Gamma:=\int_\Gamma gv\,\mathrm{d}s$. The surface Poincar\'e inequality on the compact connected curve $\Gamma$ gives $\|v\|_{L^2(\Gamma)}\lesssim |v|_{H^1(\Gamma)}$ for $v\in H^1_\diamond(\Gamma)$, so $a_\Gamma$ is coercive on $H^1_\diamond(\Gamma)$ and \eqref{eq:weak-surf} is well-posed. The coupling states that $u$ be harmonic in $\Omega$ with boundary values given by the solution of \eqref{eq:weak-surf}.

\section{Cut finite element discretization on \texorpdfstring{$\Gamma$}{Gamma}}\label{sec:cutfem}
Let $\mathcal{T}_h$ be a uniform Cartesian mesh of $\mathbb{R}^2$ with mesh size $h$, and let $\psi$ be the level-set function with $\Gamma=\{\psi=0\}$. A cell $K\in\mathcal{T}_h$ is \emph{active} if $\Gamma\cap K\neq\varnothing$; let $\mathcal{T}_h^\Gamma$ denote the active mesh, let $\gamma$ be its vertex set, and set
\[
   \Omega_h^\Gamma:=\bigcup_{K\in\mathcal T_h^\Gamma}K.
\]
Following the standard CutFEM convention, $V_h\subset H^1(\Omega_h^\Gamma)$ denotes the conforming bilinear ($\mathcal{Q}_1$) finite element space on the active mesh (equivalently, the span of those global $\mathcal{Q}_1$ basis functions whose support intersects $\Gamma$). The curve on each active cell $K$ is approximated by a straight line segment $\Gamma_K^h$ connecting the level-set zero-crossings of $\psi$ on $\partial K$, with $\Gamma^h=\bigcup_K\Gamma_K^h$. The surface restriction and the corresponding TraceFEM space are
\begin{equation}\label{eq:ambient-trace-space}
   \operatorname{Tr}_h:V_h\to H^1(\Gamma^h),\qquad
   \operatorname{Tr}_h v_h:=v_h|_{\Gamma^h},\qquad
   W_h^{\rm tr}:=\operatorname{Tr}_hV_h\subset H^1(\Gamma^h).
\end{equation}
Thus $V_h$ is the ambient active-band space, whereas $W_h^{\rm tr}$ is the actual surface trial space. The map $\operatorname{Tr}_h$ is not injective on $V_h$; this distinction separates functional well-posedness on the surface from algebraic stability of active-grid coefficients. Let $\mathcal N(x)$ denote the set of four lattice neighbours of a vertex $x$ in the five-point stencil. We partition the active vertices into three regions $\gamma=\gamma_1\,\dot\cup\,\gamma_2\,\dot\cup\,\gamma_3$, where
\begin{equation}\label{eq:g2-g3-def}
\begin{aligned}
    \gamma_1 &:= \{x\in\gamma : \psi(x)<0\},\\
    \gamma_2 &:= \{x\in\gamma : \psi(x)>0,\; \mathcal N(x)\cap\gamma_1\neq\varnothing\},\\
    \gamma_3 &:= \{x\in\gamma : \psi(x)>0\}\setminus\gamma_2,
\end{aligned}
\end{equation}
Thus $\gamma_1$ comprises the interior active vertices, $\gamma_2$ the exterior active vertices within one bulk-stencil reach of $\gamma_1$, and $\gamma_3$ the remaining exterior active vertices; see \Cref{fig:cutcell}.

\begin{assumption}[Active-grid geometry]\label{ass:geom}
The ambient mesh $\mathcal{T}_h$ is uniform Cartesian, and $V_h$ is the conforming bilinear $\mathcal{Q}_1$ space on the active mesh $\mathcal{T}_h^\Gamma$. For $h$ sufficiently small, $\Gamma^h$ is obtained by connecting the zero crossings of $\psi$ on mesh edges. We assume the standard nondegenerate cut convention: no lattice vertex lies on $\Gamma$ (a fixed sign tie-break is used in computation), all edge intersections are transverse, each active cell is cut by at most one connected interface segment, and every active cell has vertices of both signs. Let
\[
   \mathcal V_h^{\rm int}:=\{\xi\in h\mathbb Z^2:\psi(\xi)<0\},
   \qquad
   \overline{\mathcal V}_h:=\mathcal V_h^{\rm int}\cup\gamma_2.
\]
We assume that $\mathcal V_h^{\rm int}$ is finite and connected; consequently, its one-layer closure $\overline{\mathcal V}_h$ is connected. These are combinatorial active-set assumptions; no lower bound on the cut ratio is imposed. The identity between the exterior lattice boundary of $\mathcal V_h^{\rm int}$ and $\gamma_2$ is established in \Cref{lem:dmp} from the cut convention above.
\end{assumption}

For an interface-like vertex set $\gamma'$ (such as $\gamma_1$, $\gamma_2$, or $\gamma_3$), define the scaled norm
\begin{equation}\label{eq:scaled-interface-norm}
   \|v\|_{\ell^2_h(\gamma')}^2:=h\sum_{\xi\in\gamma'}|v_\xi|^2.
\end{equation}

\begin{figure}[H]
\centering
\includegraphics[width=0.6\textwidth]{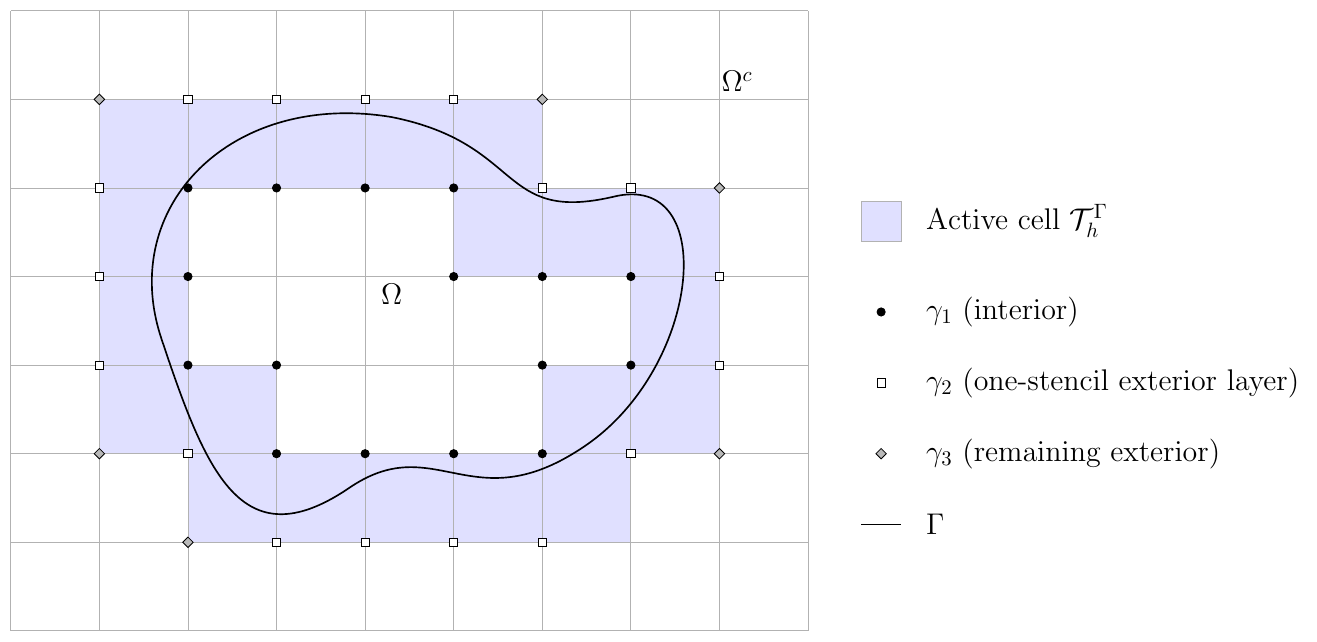}
\caption{Cut cells and point sets $\gamma_1,\gamma_2,\gamma_3$ near the interface $\Gamma$ separating the interior $\Omega$ from the exterior. $\gamma_1$ collects the interior active vertices, $\gamma_2$ the exterior active vertices within one bulk-stencil reach of $\gamma_1$, and $\gamma_3$ the remaining exterior active vertices.}\label{fig:cutcell}
\end{figure}

Let $\mathbf n_h(x)$ denote the unit outward normal to the piecewise-linear curve $\Gamma^h$, which is constant on each segment $\Gamma_K^h$. The tangential projection on each segment is $\mathcal P_h(x)=\mathbf I-\mathbf n_h(x)\mathbf n_h(x)^T$. The discrete bilinear form is assembled element-wise: for $u_h=\sum_iu_i\phi_i$ and $v_h=\sum_jv_j\phi_j$, the local stiffness and load are
\begin{equation}\label{eq:local-stiff}
    K^K_{ij} = \int_{\Gamma_K^h}(\mathcal P_h\nabla\phi_i)\cdot(\mathcal P_h\nabla\phi_j)\,\mathrm{d}s, \qquad
    b^K_i = \int_{\Gamma_K^h}g^e\,\phi_i\,\mathrm{d}s,
\end{equation}
where $g^e=g\circ p$ and $p:\Gamma^h\to\Gamma$ is the closest-point projection.
Assembling over $\mathcal{T}_h^\Gamma$ and ordering the vertices as $(\gamma_1,\gamma_2,\gamma_3)$ yields the assembled stiffness relation
\begin{equation}\label{eq:full-block}
    K \begin{pmatrix} u_1 \\ u_2 \\ u_3 \end{pmatrix} = b,
\end{equation}
with $K=K^T$. For a continuous function $U$ we write $U^h_\gamma$ for its nodal samples on a vertex set $\gamma$; the same notation applies to finite element coefficient blocks $(u_1,u_2,u_3)$.

\begin{remark}[Absence of added penalty stabilization]
The assembly \eqref{eq:local-stiff} is the plain Galerkin CutFEM discretization, using the intrinsic tangential gradient on the piecewise-linear curve $\Gamma^h$. The matrix $K$ is singular: its kernel contains all active-grid coefficient vectors whose finite element traces are constant on each component of $\Gamma^h$, and this kernel has dimension proportional to the number of active vertices. The cut-dependent pathology therefore concerns the nonzero spectrum, whose smallest positive eigenvalue can collapse as the cut degenerates. The contribution of this paper is to show that coupling \eqref{eq:full-block} with a discrete bulk harmonic extension eliminates this pathology without modifying \eqref{eq:local-stiff}.
\end{remark}

\section{Lattice Green's function and discrete harmonic extension}\label{sec:lgf-extension}

Let $L_h$ denote the standard five-point discrete Laplacian on the Cartesian lattice $h\mathbb{Z}^2$. The \emph{lattice Green's function} $G_h$ is the fundamental solution satisfying $-L_h G_h=\delta_0$ on $h\mathbb{Z}^2$ with the scaled point source $\delta_0(0)=h^{-2}$, computed to machine precision using the procedure of \citet{xia2026geometrically}. In two dimensions $G_h$ is determined only up to an additive constant. We fix the canonical Martinsson--Rodin normalization $G_h(0)=0$; the lower singular-value bound used below is stated for this normalization. The invariance of the reduced operators under any alternative normalization for which the trace block remains invertible is recorded after \Cref{prop:equivalence}.

\subsection{Single- and double-layer potentials}
Following \citet{martinsson2009boundary} and \citet{xia2026geometrically}, we construct discrete single- and double-layer potential matrices indexed by source and target points in $\gamma$:
\begin{equation}\label{eq:SD-def}
    S_{ij} = G_h(x_i-x_j),\quad
    D_{ij} = \sum_{\substack{k\sim j,\,x_k\notin\gamma}}
    \bigl[G_h(x_i-x_k)-G_h(x_i-x_j)\bigr],
\end{equation}
where $k\sim j$ denotes lattice neighbours of $x_j$. The sum in $D_{ij}$ runs over neighbours of $x_j$ lying outside the active vertex set $\gamma$ and plays the role of a one-sided discrete (outward) normal derivative of $G_h$ in the source variable. Restricting source indices to $\gamma_2$ and target indices to either $\gamma_1$ or $\gamma_2$, we obtain the blocks $S_{22},S_{12},D_{22},D_{12}$. We henceforth abbreviate
\begin{equation*}
    P_2 := S_{22}\text{ or }D_{22},\qquad P_1 := S_{12}\text{ or }D_{12},
\end{equation*}
with the single- or double-layer choice held fixed throughout each construction.

\subsection{Discrete harmonic extension operator}
\begin{definition}[Discrete harmonic extension]\label{def:H}
When the block $P_2$ is invertible, 
the discrete harmonic extension operator $H$ from $\gamma_2$ to $\gamma_1$ is 
\begin{equation}\label{eq:H-def}
    H := P_1P_2^{-1} : \mathbb{R}^{|\gamma_2|}\to\mathbb{R}^{|\gamma_1|}.
\end{equation}
\end{definition}

A lattice function discrete-harmonic in a neighbourhood of $\gamma$ admits a layer-potential representation with density on $\gamma_2$: solving $P_2\mu=u^h_{\gamma_2}$ recovers the density and $P_1\mu$ gives the interior trace; composing yields $P_1P_2^{-1}$. Both single- and double-layer choices give a valid $H$ with second-order consistency.

\begin{remark}[Invertibility of $P_2$]\label{rem:P2-invertible}
For the single-layer formulation in two dimensions, Theorem~6.1 of \citet{martinsson2009boundary} proves, for the canonical normalization $G_h(0)=0$, that the single-layer block is uniformly invertible on the node boundary of a finite connected lattice domain, with
\[
    \sigma_{\min}(P_2)\ge\tfrac18,
\]
which is independent of $h$. For the double-layer construction we separately assume uniform invertibility of $D_{22}$; no general curved-cut theorem is claimed. 
\end{remark}

\subsection{Discrete Dirichlet interpretation and stability}

\begin{lemma}[Lattice boundary and discrete maximum principle]\label{lem:dmp}
Let
\[
\partial_h\mathcal V_h^{\rm int}:=
\{\eta\notin\mathcal V_h^{\rm int}:\eta\sim\xi
\text{ for some }\xi\in\mathcal V_h^{\rm int}\}.
\]
Then $\partial_h\mathcal V_h^{\rm int}=\gamma_2$. If $w$ is discrete harmonic on $\mathcal V_h^{\rm int}$, then
\begin{equation}\label{eq:dmp}
    \|w\|_{\ell^\infty(\mathcal V_h^{\rm int})}\le\|w\|_{\ell^\infty(\gamma_2)}.
\end{equation}
In particular, the discrete Dirichlet problem on $\mathcal V_h^{\rm int}$ with data on $\gamma_2$ is unique.
\end{lemma}

\begin{proof}
If $\eta\in\gamma_2$, it is exterior and adjacent to a vertex of $\gamma_1\subset\mathcal V_h^{\rm int}$, hence $\eta\in\partial_h\mathcal V_h^{\rm int}$. Conversely, let $\eta\in\partial_h\mathcal V_h^{\rm int}$ and choose an interior neighbour $\xi$. The edge $\xi\eta$ changes sign; the cells incident on that edge are therefore active. Thus $\xi\in\gamma_1$, $\eta\in\gamma$, and \eqref{eq:g2-g3-def} gives $\eta\in\gamma_2$. This proves equality. If the maximum is attained at an interior vertex $\xi$, $L_hw(\xi)=0$ gives $w(\xi)=\tfrac14\sum_{\eta\sim\xi}w(\eta)$; equality forces all neighbours to share the maximum. Propagation along edges shows the maximum is attained on $\partial_h\mathcal V_h^{\rm int}=\gamma_2$. Applying the same argument to $-w$ gives \eqref{eq:dmp}. Applying the bound to the difference of two solutions with identical boundary data proves uniqueness.
\end{proof}

\begin{lemma}[Dirichlet interpretation of the layer-potential extension]\label{lem:H-Dirichlet}
For either the single-layer choice, or the double-layer choice under \Cref{rem:P2-invertible}, $Hu_2=P_1P_2^{-1}u_2$ is the restriction to $\gamma_1$ of the unique discrete harmonic function on $\mathcal V_h^{\rm int}$ with Dirichlet data $u_2$ on $\gamma_2$. Consequently,
\[
 H\mathbf1_{\gamma_2}=\mathbf1_{\gamma_1},
 \qquad
 \|Hu_2\|_{\ell^\infty(\gamma_1)}
 \le \|u_2\|_{\ell^\infty(\gamma_2)}.
\]
\end{lemma}

\begin{proof}
Put $\mu=P_2^{-1}u_2$ and let $z=P\mu$ be the corresponding full lattice potential. For the single layer, $-L_hz$ is supported on $\gamma_2$, so $z$ is discrete harmonic on $\mathcal V_h^{\rm int}$. For the double layer, applying $-L_h$ to each difference in \eqref{eq:SD-def} produces point sources on $\gamma_2$ and on its neighbours outside $\gamma$. Every interior neighbour of a point in $\gamma_2$ lies in $\gamma_1$ by the sign-changing-edge argument above, and hence is not outside $\gamma$; the additional sources are therefore exterior to $\mathcal V_h^{\rm int}$. Thus the double-layer potential is also harmonic on $\mathcal V_h^{\rm int}$. In either case $z|_{\gamma_2}=P_2\mu=u_2$. Uniqueness in \Cref{lem:dmp} identifies $z$ with the discrete Dirichlet solution, and restriction to $\gamma_1$ gives $z|_{\gamma_1}=P_1\mu=Hu_2$. Taking constant boundary data and using uniqueness gives constant reproduction, while \eqref{eq:dmp} gives the $\ell^\infty$ bound.
\end{proof}

\begin{remark}[Companion finite element harmonic lift]\label{rem:companion-P1}
The five-point harmonic extension has an exact finite element interpretation. Subdivide each Cartesian square into two right triangles and let $X_h$ be the continuous piecewise-affine space on this companion triangulation. For an interior lattice vertex, the $X_h$ stiffness equation is, up to a positive scaling, the five-point relation: the cotangent weight of the square diagonal vanishes and the four axial weights remain. Consequently, the lattice Dirichlet solution in \Cref{lem:H-Dirichlet} is the nodal vector of the discrete harmonic lift in $X_h$. The surface method still uses the bilinear space $V_h$: the same nodal vector is transferred to its $\mathcal Q_1$ representative before restriction to $\Gamma^h$. Thus $X_h$ and $V_h$ provide companion finite element realizations of the same constrained nodal data, not the same cellwise function. This viewpoint allows energy-minimization and finite element lifting arguments to be used independently of the layer-potential representation.
\end{remark}

\begin{lemma}[$\ell^2_h$-stability of $H$]\label{lem:H-stab}
Under \Cref{ass:geom}, assume that $P_2$ is invertible for the chosen layer-potential construction, so that $H=P_1P_2^{-1}$ is well defined. Then there exists $C_H>0$ independent of $h$ and of the position of $\Gamma$ relative to $\mathcal{T}_h$ such that, for every $u_2\in\mathbb{R}^{|\gamma_2|}$,
\begin{equation}\label{eq:H-stab}
    \|Hu_2\|_{\ell^2_h(\gamma_1)}\le C_H\|u_2\|_{\ell^2_h(\gamma_2)}.
\end{equation}
\end{lemma}

\begin{proof}
By \Cref{lem:H-Dirichlet}, $Hu_2=w|_{\gamma_1}$, where $w$ is discrete
harmonic on $I:=\mathcal V_h^{\rm int}$ and $w|_B=u_2$ on $B:=\gamma_2$.
Let
\[
 \mathcal E_I:=\bigl\{e=\{\zeta,\zeta'\}:\zeta\sim\zeta',\
       \{\zeta,\zeta'\}\cap I\ne\varnothing\bigr\},
 \qquad
 \delta_ev:=v(\zeta)-v(\zeta'),
 \qquad
 \mathcal D_h[v]:=\frac12\sum_{e\in\mathcal E_I}|\delta_ev|^2.
\]
Every $\xi\in\gamma_1$ shares an active cell with an exterior vertex, so
$\operatorname{dist}_{I\cup B}(\xi,B)\le D_0:=2$. By its Euler--Lagrange
equations, $w$ minimizes $\mathcal D_h$ among all extensions of $u_2$. With
$d(\zeta):=\operatorname{dist}_{I\cup B}(\zeta,B)$ and a nearest
$\eta(\zeta)\in B$, the comparison function
\[
 \widetilde w(\zeta):=
 \begin{cases}
   2^{-d(\zeta)}u_2\bigl(\eta(\zeta)\bigr),&d(\zeta)\le D_0,\\
   0,&d(\zeta)>D_0.
 \end{cases}
\]
has trace $u_2$ on $B$. Since $D_0$ and the lattice degree are fixed, each
boundary value enters only uniformly many edge differences; hence
\begin{equation}\label{eq:H-test-energy}
   \mathcal D_h[w]\le\mathcal D_h[\widetilde w]
   \le C\sum_{\eta\in B}|u_2(\eta)|^2.
\end{equation}

For each $\xi\in\gamma_1$, choose a shortest path $\Pi(\xi)$ to
$\eta(\xi)\in B$; its length is at most $D_0$. Telescoping, followed by
$|a+b|^2\le2|a|^2+2|b|^2$ and Cauchy--Schwarz on the jump sum, gives
\[
 |w(\xi)|^2
 \le 2|u_2(\eta(\xi))|^2
      +2D_0\sum_{e\in\Pi(\xi)}|\delta_ew|^2.
\]
Each boundary vertex and each edge occurs in only uniformly many such paths.
Summing and using \eqref{eq:H-test-energy} therefore yields
$\sum_{\gamma_1}|w|^2\le C\sum_B|u_2|^2$. Multiplication by $h$ proves
\eqref{eq:H-stab}; all constants are cut independent.
\end{proof}

\subsection{Local extrapolation}\label{sec:extrap-construction}
The active CutFEM surface space may contain exterior vertices in $\gamma_3$, determined from $\gamma_1\cup\gamma_2$ by a local extrapolation. For $\eta\in\gamma_3$ and an admissible direction $e\in\{\pm e_1,\pm e_2\}$ with $\eta-he,\eta-2he\in\gamma_1\cup\gamma_2$, set
\begin{equation}\label{eq:3pt-extrap}
    v(\eta)=2v(\eta-he)-v(\eta-2he),
\end{equation}
averaging over two orthogonal admissible directions when available.

\begin{assumption}[Local extrapolation stencils]\label{ass:extrap-stencil}
Every $\eta\in\gamma_3$ has at least one admissible reconstruction of the form \eqref{eq:3pt-extrap}. The associated stencil patches have diameter $O(h)$ and uniformly finite overlap, independently of the cut position. In implementation, a uniformly local least-squares affine extrapolation may be substituted provided that every stencil uses a uniformly bounded number of nodes with uniformly bounded weights, is exact on affine functions, and retains uniformly finite patch overlap.
\end{assumption}

Writing the chosen relations in the form
\begin{equation}\label{eq:extrap}
   R_1u_1+R_2u_2+u_3=0
\end{equation}
with $R_1\in\mathbb{R}^{|\gamma_3|\times|\gamma_1|}$ and $R_2\in\mathbb{R}^{|\gamma_3|\times|\gamma_2|}$, defines sparse matrices $R_1$ and $R_2$ with a uniformly bounded number of entries per row.

\begin{proposition}[Stability and consistency]\label{prop:extrap-linear}
Under \Cref{ass:geom,ass:extrap-stencil}, and assuming that $P_2$ is invertible for the chosen layer-potential construction, there is a constant $C$ independent of $h$ and the cut configuration such that, with $J=R_1H+R_2$,
\begin{enumerate}
\item[(i)] $\|Jv_2\|_{\ell^2_h(\gamma_3)}\le C\|v_2\|_{\ell^2_h(\gamma_2)}$ for all $v_2\in\mathbb{R}^{|\gamma_2|}$;
\item[(ii)] for every $U\in W^{3,\infty}$ in a fixed tubular neighbourhood of $\Gamma$, $\bigl\|U^h_{\gamma_3}+R_1U^h_{\gamma_1}+R_2U^h_{\gamma_2}\bigr\|_{\ell^\infty(\gamma_3)}\le Ch^2\|U\|_{W^{3,\infty}}$;
\item[(iii)] constants are reproduced: $R_1\mathbf 1_{\gamma_1}+R_2\mathbf 1_{\gamma_2}+\mathbf 1_{\gamma_3}=0$, hence $J\mathbf 1_{\gamma_2}=-\mathbf 1_{\gamma_3}$ whenever $H\mathbf 1_{\gamma_2}=\mathbf 1_{\gamma_1}$.
\end{enumerate}
\end{proposition}

\begin{proof}
\emph{(i)} Let $A=[R_1\mid R_2]$ and $z=(Hv_2,v_2)$. Uniformly bounded
row size, coefficients, and column overlap from \Cref{ass:extrap-stencil} give
\[
 \|Jv_2\|_{\ell_h^2(\gamma_3)}^2=\|Az\|_{\ell_h^2(\gamma_3)}^2
 \lesssim \|Hv_2\|_{\ell_h^2(\gamma_1)}^2
          +\|v_2\|_{\ell_h^2(\gamma_2)}^2.
\]
Now apply \Cref{lem:H-stab}.

\emph{(ii)} For an explicit one-direction stencil, Taylor expansion gives
\[
 U(\eta)-2U(\eta-he)+U(\eta-2he)
 =h^2\partial_e^2U(\eta)
   +O\bigl(h^3\|U\|_{W^{3,\infty}}\bigr),
\]
so averaging preserves the $O(h^2)$ bound. For the residual $L_\eta$ of an
admissible substitute, subtract the first-order Taylor polynomial at $\eta$:
affine exactness cancels it, while the $O(h)$ patch diameter and bounded weights give
$|L_\eta(U)|\lesssim h^2\|U\|_{W^{2,\infty}}$. Taking the maximum proves~(ii).

\emph{(iii)} Exactness on constants gives
\[
 R_1\mathbf1_{\gamma_1}+R_2\mathbf1_{\gamma_2}
 +\mathbf1_{\gamma_3}=0.
\]
Together with $H\mathbf1_{\gamma_2}=\mathbf1_{\gamma_1}$ from
\Cref{lem:H-Dirichlet}, this implies
\[
 J\mathbf1_{\gamma_2}
 =-\mathbf1_{\gamma_3}.
\]

\end{proof}

\section{Reduced trial space and coupled formulation}\label{sec:reduced}

\subsection{Discrete extrapolation and harmonic constraint}

We now obtain two linear relations constraining $(u_1,u_2,u_3)$ to a subspace parametrized by $u_2$ alone. The first is the harmonic extension
\begin{equation}\label{eq:harm-12}
    u_1 = Hu_2,\qquad H=P_1P_2^{-1};
\end{equation}
the second is the extrapolation relation~\eqref{eq:extrap}. Combining
\eqref{eq:harm-12} and \eqref{eq:extrap} gives
\begin{equation}\label{eq:u3-from-u2}
    u_3=-Ju_2,\qquad J:=R_1H+R_2.
\end{equation}

Introduce the boundary-data space $\mathcal M_h:=\mathbb R^{|\gamma_2|}$, equipped when needed with the scaled norm $\ell_h^2(\gamma_2)$. At coefficient level define
\begin{equation}\label{eq:E-def}
E := \begin{pmatrix} H\\I\\-J\end{pmatrix}:\mathbb{R}^{|\gamma_2|}\to\mathbb{R}^{|\gamma|},\qquad\begin{pmatrix}u_1\\u_2\\u_3\end{pmatrix}=Eu_2.
\end{equation}
If $\{\phi_\xi\}_{\xi\in\gamma}$ is the nodal basis of $V_h$, the corresponding finite element extension operator is
\begin{equation}\label{eq:Eh-def}
   \mathcal E_h:\mathcal M_h\to V_h,\qquad
   \mathcal E_hv_2:=\sum_{\xi\in\gamma}(Ev_2)_\xi\phi_\xi.
\end{equation}
Because the $\gamma_2$ block of $E$ is the identity, $\mathcal E_h$ is injective.

\begin{definition}[Reduced active and trace spaces]\label{def:Vred3}
The reduced active finite element space and its surface trace space are
\begin{align}
   V_h^{\mathrm{red}}
   &:=\mathcal E_h\mathcal M_h
     =\bigl\{u_h\in V_h:u_1=Hu_2,\;u_3=-Ju_2\bigr\}\subset V_h,
     \label{eq:Vred-def}\\
   W_h^{\mathrm{red}}
   &:=\operatorname{Tr}_hV_h^{\mathrm{red}}
     \subset W_h^{\rm tr}\subset H^1(\Gamma^h).
     \label{eq:Wred-def}
\end{align}
\end{definition}

Equivalently, under the nodal coefficient isomorphism $V_h\simeq\mathbb R^{|\gamma|}$ and with coefficient ordering $(\gamma_1,\gamma_2,\gamma_3)$,
\begin{equation}\label{eq:constraint-kernel}
   V_h^{\mathrm{red}}\simeq\ker C_h=\operatorname{range}(E),\qquad
   C_h:=
   \begin{pmatrix}
      I&-H&0\\
      0&J&I
   \end{pmatrix}.
\end{equation}
The trace map need not be injective on the full ambient space $V_h$. Its quantitative injectivity on $V_h^{\rm red}$ is precisely the boundary-layer observability property of \Cref{lem:observability}; convergence of the surface trace and conditioning of its coefficient parametrization are therefore kept logically distinct.

\subsection{Symmetric Galerkin reduction}

We seek a trace $w_h\in W_h^{\mathrm{red}}$ satisfying the surface Galerkin identity on $W_h^{\mathrm{red}}$. Equivalently, choose an active representative $\mathcal E_hu_2\in V_h^{\mathrm{red}}$ and test with representatives $\mathcal E_hv_2$. Substitution at coefficient level yields
\begin{equation}\label{eq:reduced}
    K_{\mathrm{red}}u_2=b_{\mathrm{red}}, \quad K_{\mathrm{red}}:=E^TKE, \quad b_{\mathrm{red}}:=E^Tb.
\end{equation}

Thus $K_{\rm red}$ is symmetric by construction. The left multiplication by $E^T$ is essential: substituting only the trial relation would give an overdetermined residual system rather than a Galerkin method.

Constant reproduction in \Cref{lem:H-Dirichlet,prop:extrap-linear}(iii) gives
\begin{equation}\label{eq:constant-mode}
 E\mathbf1_{\gamma_2}=\mathbf1_\gamma,
 \qquad
 K_{\rm red}\mathbf1_{\gamma_2}=0,
\end{equation}
because the tangential gradient of the constant finite element function vanishes on $\Gamma^h$.

Accordingly, the constant mode is fixed by imposing
\begin{equation}\label{eq:discrete-gauge}
    \int_{\Gamma^h}\operatorname{Tr}_h(\mathcal E_hu_2)\,ds_h=0
\end{equation}
or, equivalently, by working on the quotient by constants. Since
$K_{\rm red}$ is symmetric and has the constant null vector, every reduced
right-hand side $r$ is replaced, when necessary, by its compatible projection
\[
 r^{\diamond}:=r
 -\frac{\mathbf1^Tr}{\mathbf1^T\mathbf1}\,\mathbf1 ,
\qquad \mathbf1^Tr^{\diamond}=0.
\]
For an inhomogeneous bulk equation $-\Delta u=f$, take $u^p=G_h*f_h$ as a particular lattice solution of $-L_hu^p=f_h$ (with $f_h$ the active-lattice sampling of $f$, having sufficient decay or compact support); the residual $w:=u-u^p$ is discrete-harmonic in the interior, so the reduction applies to $w$,
yielding
\begin{equation}\label{eq:w2-inhomogeneous}
    E^TKEw_2=E^Tb-E^TKu^p_\gamma .
\end{equation}
The trial set is the affine space $u^p_\gamma+\operatorname{range}(E)$. The particular solution need not satisfy the $\gamma_3$ extrapolation identity exactly; by \Cref{prop:extrap-linear}(ii), its residual is $O(h^2\|u^p\|_{W^{3,\infty}})$ and is of the same order as the extrapolation consistency error already retained. The corrected right-hand side in \eqref{eq:w2-inhomogeneous} is projected onto $\mathbf1^\perp$ by the formula above before the gauge is imposed.

\subsection{Density-parametrized formulation}\label{sec:density}

The direct formulation requires the inverse $P_2^{-1}$. A density-based variant avoids this by parametrizing $V_h^{\mathrm{red}}$ through a lattice density $q\in\mathbb{R}^{|\gamma_2|}$: with $u_1=P_1q$ and $u_2=P_2q$, the extrapolation relation gives $u_3=-(R_1P_1+R_2P_2)q$, and hence
\begin{equation}\label{eq:F-def}
    F:=\begin{pmatrix}P_1\\P_2\\-(R_1P_1+R_2P_2)\end{pmatrix}:\mathbb{R}^{|\gamma_2|}\to\mathbb{R}^{|\gamma|}.
\end{equation}
The density Galerkin system is
\begin{equation}\label{eq:reduced-density}
    \widetilde K_{\mathrm{red}}q=\widetilde b_{\mathrm{red}}, \quad \widetilde K_{\mathrm{red}}:=F^TKF, \quad \widetilde b_{\mathrm{red}}:=F^Tb.
\end{equation}

\begin{proposition}[Algebraic equivalence]\label{prop:equivalence}
If $P_2$ is invertible and $H=P_1P_2^{-1}$, then $F=EP_2$ and $\widetilde K_{\mathrm{red}}=P_2^TK_{\mathrm{red}}P_2$, $\widetilde b_{\mathrm{red}}=P_2^Tb_{\mathrm{red}}$. Consequently $q$ solves \eqref{eq:reduced-density} if and only if $u_2:=P_2q$ solves \eqref{eq:reduced}, and $Fq=Eu_2$.
\end{proposition}

\begin{proof}
Since $H=P_1P_2^{-1}$, $P_1=HP_2$, hence $F=\bigl(HP_2,P_2,-(R_1H+R_2)P_2\bigr)^T=EP_2$. The two congruence identities follow.
\end{proof}

\begin{remark}[LGF normalization invariance]\label{rem:LGF-normalization}
Replacing the two-dimensional LGF by $G_h+c$ adds $c\,\mathbf 1_\gamma\mathbf 1_{\gamma_2}^T$ to the single-layer reconstruction $F$. Since $K\mathbf 1_\gamma=0$ and $K=K^T$, $F^TKF$ is unchanged. If the shifted block $S_{22}+c\mathbf1\mathbf1^T$ is invertible, constant reproduction also gives
\[
 (S_{12}+c\mathbf1\mathbf1^T)(S_{22}+c\mathbf1\mathbf1^T)^{-1}
 =S_{12}S_{22}^{-1},
\]
so the direct harmonic extension is unchanged as well. An arbitrary additive shift can, however, make the rank-one-updated trace block singular; this is why the canonical normalization is fixed before invoking the invertibility theorem.
\end{remark}

The direct and density formulations generate the same reduced trial space whenever $P_2$ is invertible. They are not algebraically equivalent as preconditioned linear systems: $\widetilde K_{\mathrm{red}}=P_2^TK_{\mathrm{red}}P_2$ is a congruence, not a similarity, and $P_2$ is not an isometry. The spectrum and condition number may differ substantially. This distinction is essential for the single-layer formulation: the single-layer trace operator has order $-1$, so the congruence by $P_2=S_{22}$ changes the effective algebraic order of the surface stiffness operator and acts as an operator preconditioner. Computationally, the density formulation avoids forming $P_2^{-1}$ explicitly, is amenable to a matrix-free implementation, and preserves Galerkin symmetry.

The inhomogeneity is handled similarly: with $u_\gamma=u^p_\gamma+Fq$ and homogeneous test functions $v_\gamma=F\mu$,
\begin{equation}\label{eq:reduced-density-inhomog}
    \widetilde K_{\mathrm{red}}q = \widetilde b_{\mathrm{red}}-F^TKu^p_\gamma.
\end{equation}
As in the direct formulation, the affine trial set $u^p_\gamma+\operatorname{range}(F)$ leaves only the $O(h^2)$ extrapolation residual of the particular solution, and the corrected load is projected onto the exact range before gauge fixing.

\section{\texorpdfstring{Approximation and error analysis}{Approximation and error analysis}}\label{sec:analysis}

This section establishes a priori approximation and error estimates for the reduced formulation, following the discrete-extension framework~\citep{burman2022cutfem}. In the notation of \Cref{def:Vred3}, the construction has the embedding structure
\begin{equation}\label{eq:extension-space-chain}
   \mathcal M_h\xrightarrow{\;\mathcal E_h\;}
   V_h^{\rm red}\hookrightarrow V_h,\qquad
   W_h^{\rm red}=\operatorname{Tr}_hV_h^{\rm red}
   \hookrightarrow W_h^{\rm tr}.
\end{equation}
The analysis separates four properties of this chain: left-invertibility on the master block and constant reproduction, upper stability of $\mathcal E_h$, approximation by $W_h^{\rm red}$, and lower trace stability (observability) of $\operatorname{Tr}_h\mathcal E_h$. Constant reproduction and reduced-space approximation drive the surface error analysis. The upper and lower stability bounds identify surface traces with coefficient vectors and control the reduced operator. The distinguishing feature from polynomial discrete extension is that $\mathcal E_h$ combines the bulk discrete harmonic extension $H=P_1P_2^{-1}$ with the extrapolation \eqref{eq:extrap}, rather than extending from a neighbouring interior cell or adding a stabilization term.

\subsection{Analytical assumptions}\label{sec:assumptions}

Under the active-grid geometry of \Cref{ass:geom}, the closest-point projection $p:\Gamma^h\to\Gamma$ satisfies the standard geometric consistency estimates. For $v$ on $\Gamma$, write $v^e:=v\circ p$. Then $\|p-\mathrm{id}\|_{L^\infty(\Gamma^h)}\lesssim h^2$, $\|\mathbf n\circ p-\mathbf n_h\|_{L^\infty(\Gamma^h)}\lesssim h$, and $\|1-ds/ds_h\|_{L^\infty(\Gamma^h)}\lesssim h^2$~\citep{burman2015stabilized,grande2018analysis,olshanskii2009finite}.

\begin{assumption}[Data and regularity]\label{ass:data}
The data satisfy $g\in L^2(\Gamma)$, $\int_\Gamma g\,ds=0$, and the solution is normalized by $\int_\Gamma u\,ds=0$. Moreover, $u\in H^2(\Gamma)$. The harmonic lift $U$ of $u$ (the unique harmonic extension of $u$ to $\Omega$) is assumed to satisfy $U\in W^{4,\infty}(\Omega\cap U_\Gamma)$ on the interior side of a fixed tubular neighbourhood $U_\Gamma$ of $\Gamma$. Interior estimates for the harmonic function control the compact set $\Omega\setminus U_\Gamma$ by this collar norm. Since $\Omega$ is smooth, let
\[
 \mathfrak E:W^{4,\infty}(\Omega)\longrightarrow W^{4,\infty}(\mathbb R^2)
\]
be a bounded Sobolev extension operator, for example Stein's extension operator, satisfying $(\mathfrak EV)|_\Omega=V$. Write $\widetilde U:=\mathfrak EU$. Then
\[
 \|\widetilde U\|_{W^{4,\infty}(\mathbb R^2)}
 \lesssim\|U\|_{W^{4,\infty}(\Omega)}
 \lesssim\|U\|_{W^{4,\infty}(\Omega\cap U_\Gamma)}.
\]
Only $\widetilde U$ is sampled at exterior stencil nodes; it is not assumed harmonic outside $\Omega$. This regularity is used for nodal interpolation, pointwise LGF consistency, and local extrapolation consistency. Global boundary regularity gives $U\in H^{5/2}(\Omega)$ for $H^2(\Gamma)$ Dirichlet data; the stronger one-sided $W^{4,\infty}$ hypothesis is retained only for the pointwise approximation result.
\end{assumption}

\subsection{Mean-zero spaces and harmonic-extension consistency}
\label{sec:harmonic-consistency}

The mean-zero reduced active and trace spaces are
\begin{equation}\label{eq:Vreddiamond}
\begin{aligned}
   V_{h,\diamond}^{\rm red}
   &:= \left\{v_h\in V_h^{\rm red}:\int_{\Gamma^h}\operatorname{Tr}_hv_h\,ds_h=0\right\},\\
   W_{h,\diamond}^{\rm red}
   &:=\operatorname{Tr}_hV_{h,\diamond}^{\rm red}
     =W_h^{\rm red}\cap H^1_\diamond(\Gamma^h).
\end{aligned}
\end{equation}
The functional Galerkin and error analysis is posed on $W_{h,\diamond}^{\rm red}$, where the intrinsic surface norm is a genuine norm. Statements about the active representative or its coefficient vector additionally invoke observability when injectivity of $\operatorname{Tr}_h\mathcal E_h$ is required.

\begin{proposition}[Consistency of the LGF harmonic extension]\label{prop:H-consistency}
Under \Cref{ass:geom,ass:data}, assume that $P_2$ is invertible for the chosen layer-potential construction. Let $U$ be harmonic in $\Omega$ and satisfy the one-sided regularity of \Cref{ass:data}, and let $\widetilde U=\mathfrak E U$.
\begin{align}\label{eq:H-consistency}
    \|\widetilde U^h_{\gamma_1}-H\widetilde U^h_{\gamma_2}\|_{\ell^\infty(\gamma_1)}
    \le Ch^2\|\widetilde U\|_{W^{4,\infty}(\mathbb R^2)},
\end{align}
with $C$ independent of $h$ and the cut configuration.
\end{proposition}

\begin{proof}
Let $w$ be the discrete Dirichlet extension of
$\widetilde U^h|_{\gamma_2}$ and set $e=w-\widetilde U^h$. Then
$w|_{\gamma_1}=H\widetilde U^h_{\gamma_2}$ by \Cref{lem:H-Dirichlet}, while
$\Delta\widetilde U(\xi)=0$ for $\xi\in\mathcal V_h^{\rm int}$. The
five-point truncation estimate therefore gives
\[
 e|_{\gamma_2}=0,
 \qquad
 \|L_he\|_{\ell^\infty(\mathcal V_h^{\rm int})}
 \le C_0h^2\|\widetilde U\|_{W^{4,\infty}(\mathbb R^2)}.
\]
Choose fixed $x_0$ and $R$, independent of $h$ and the cut, such that
$\overline{\mathcal V}_h\subset B_R(x_0)$, and set
\[
 \psi(x):=M\bigl(R^2-|x-x_0|^2\bigr),
 \qquad
 M:=\frac{C_0}{4}h^2\|\widetilde U\|_{W^{4,\infty}(\mathbb R^2)}.
\]
Since $L_h\psi=-4M$, one has $L_h(\psi\pm e)\le0$ in
$\mathcal V_h^{\rm int}$ and $\psi\pm e\ge0$ on $\gamma_2$. The discrete
minimum principle yields $|e|\le\psi\le MR^2$; restriction to $\gamma_1$
proves \eqref{eq:H-consistency}. No harmonic continuation across $\Gamma$ is
used; the extension only supplies values for stencils crossing $\Gamma$.
\end{proof}

\subsection{Approximation by the reduced space}
\label{sec:approx-Vred}

\begin{lemma}[Approximation by the reduced space]\label{lem:trace-approx}
Under \Cref{ass:geom,ass:extrap-stencil,ass:data}, assume that $P_2$ is invertible for the chosen layer-potential construction. Then there exist $v_{\star,2}\in\mathcal M_h$ and
$w_\star:=\operatorname{Tr}_h(\mathcal E_hv_{\star,2})\in W_h^{\rm red}$ such that
\begin{equation}\label{eq:trace-approx}
   \|u^e-w_\star\|_{L^2(\Gamma^h)}
   +h\|u^e-w_\star\|_{H^1(\Gamma^h)}
   \lesssim h^2\bigl(\|u\|_{H^2(\Gamma)}+\|U\|_{W^{4,\infty}(\Omega\cap U_\Gamma)}\bigr),
\end{equation}
where $u^e = u \circ p$, $U$ is the harmonic lift, and exterior nodal values are those of $\widetilde U=\mathfrak E U$. Moreover, subtracting the average of $w_\star$ over $\Gamma^h$ gives an element of $W_{h,\diamond}^{\rm red}$ satisfying the same estimate.
\end{lemma}

\begin{proof}
Set $\mathcal N:=\|u\|_{H^2(\Gamma)}+
\|U\|_{W^{4,\infty}(\Omega\cap U_\Gamma)}$ and
$\|z\|_{*,h}:=\|z\|_{L^2(\Gamma^h)}+h\|z\|_{H^1(\Gamma^h)}$.
Let $I_h\widetilde U$ be the active-mesh interpolant, set
$v_{\star,2}:=\widetilde U^h_{\gamma_2}$, and put
$d_h:=I_h\widetilde U-\mathcal E_hv_{\star,2}$. Then $d_h|_{\gamma_2}=0$
and \Cref{prop:H-consistency} gives
$\|d_h\|_{\ell^\infty(\gamma_1)}\lesssim h^2\mathcal N$. Moreover,
\[
 d_h|_{\gamma_3}
 =\bigl(\widetilde U^h_{\gamma_3}+R_1\widetilde U^h_{\gamma_1}
        +R_2\widetilde U^h_{\gamma_2}\bigr)-R_1d_h|_{\gamma_1}.
\]
Thus \Cref{prop:extrap-linear}(ii), the bounded row weights of $R_1$, and
full-cell $\mathcal Q_1$ trace and inverse estimates yield
\[
 \|d_h\|_{\ell^\infty(\gamma)}\lesssim h^2\mathcal N,
 \qquad
 \|\operatorname{Tr}_hd_h\|_{*,h}\lesssim h^2\mathcal N.
\]
Writing $w_\star=\operatorname{Tr}_h(\mathcal E_hv_{\star,2})$, split on
$\Gamma^h$
\[
 u^e-w_\star=(u^e-\widetilde U)
 +(\widetilde U-I_h\widetilde U)+\operatorname{Tr}_hd_h.
\]
The closest-point estimates preceding \Cref{ass:data}, a normal-coordinate
Taylor expansion, and $\mathcal Q_1$ interpolation give
\[
 \|u^e-\widetilde U\|_{*,h}
 +\|\widetilde U-I_h\widetilde U\|_{*,h}\lesssim h^2\mathcal N,
\]
where all functions are restricted to $\Gamma^h$ and boundedness of
$\mathfrak E$ was used. This proves \eqref{eq:trace-approx}.

Finally, let $\overline w_\star:=|\Gamma^h|^{-1}\int_{\Gamma^h}w_\star\,ds_h$.
Since $\int_\Gamma u\,ds=0$ and the surface-measure error is $O(h^2)$,
$|\Gamma^h|^{-1}|\int_{\Gamma^h}u^e\,ds_h|\lesssim h^2\mathcal N$; hence
\eqref{eq:trace-approx} and Cauchy--Schwarz give
$|\overline w_\star|\lesssim h^2\mathcal N$. Constants belong to
$W_h^{\rm red}$, so
$w_\star-\overline w_\star\in W_{h,\diamond}^{\rm red}$, and subtracting this
constant preserves the stated estimate.
\end{proof}

\subsection{Coercivity and consistency}
\label{sec:coercivity}

For $w_h,z_h\in W_h^{\rm tr}$ define the intrinsic surface forms
\[
a_h(w_h,z_h):=\int_{\Gamma^h}\nabla_{\Gamma^h}w_h\cdot\nabla_{\Gamma^h}z_h\,ds_h ,\quad
\ell_h(z_h):=\int_{\Gamma^h}g^ez_h\,ds_h,
\]
where $g^e = g\circ p$. If $w_h=\operatorname{Tr}_h\widehat w_h$ and $z_h=\operatorname{Tr}_h\widehat z_h$ for active representatives in $V_h$, then $\nabla_{\Gamma^h}w_h=\mathcal P_h\nabla\widehat w_h$ and similarly for $z_h$, so this definition agrees with the assembled form \eqref{eq:local-stiff} and is independent of the chosen representatives. In particular, $a_h(w_h,w_h)=|w_h|_{H^1(\Gamma^h)}^2$. The method of \Cref{sec:reduced} is equivalent at trace level to the Galerkin problem
\begin{equation}\label{eq:reduced-galerkin}
    \text{find }u_h\in W_{h,\diamond}^{\mathrm{red}}:\quad a_h(u_h,v_h)=\ell_h(v_h),\quad\forall v_h\in W_{h,\diamond}^{\mathrm{red}}.
\end{equation}
By Cauchy--Schwarz and the closest-point norm equivalence,
\[
 |a_h(w_h,z_h)|\le\|w_h\|_{H^1(\Gamma^h)}\|z_h\|_{H^1(\Gamma^h)},
 \qquad
 |\ell_h(z_h)|\lesssim\|g\|_{L^2(\Gamma)}\|z_h\|_{H^1(\Gamma^h)}.
\]
Thus the form and load are uniformly continuous. The trace problem is uniquely
solvable by the coercivity estimate below. Uniqueness of the parameter vector
representing $u_h$ is a separate algebraic question governed by
\Cref{lem:observability}.

\begin{lemma}[Uniform Poincar\'e inequality and dual regularity]\label{lem:surface-estimates}
Let $S_h=|\Gamma^h|$.
\begin{enumerate}
\item[(i)] Every $v\in H^1_\diamond(\Gamma^h)$ satisfies
\[
 \|v\|_{L^2(\Gamma^h)}\le\frac{S_h}{2\pi}|v|_{H^1(\Gamma^h)}.
\]
Consequently,
\begin{equation}\label{eq:coercivity}
 a_h(v_h,v_h)\gtrsim\|v_h\|_{H^1(\Gamma^h)}^2,
 \qquad v_h\in W_{h,\diamond}^{\rm red}.
\end{equation}
\item[(ii)] If $f_h\in L^2(\Gamma^h)$ has zero mean and $\phi_h\in H^1_\diamond(\Gamma^h)$ solves
\[
 \int_{\Gamma^h}\nabla_{\Gamma^h}z\cdot\nabla_{\Gamma^h}\phi_h\,ds_h
 =\int_{\Gamma^h}f_hz\,ds_h,
 \qquad z\in H^1_\diamond(\Gamma^h),
\]
then
\begin{equation}\label{eq:polygon-dual-regularity}
 \|\phi_h\|_{H^2(\Gamma^h)}\lesssim\|f_h\|_{L^2(\Gamma^h)}.
\end{equation}
Here $H^2(\Gamma^h)$ is the periodic Sobolev space defined through the
arclength parametrization of $\Gamma^h$.
\end{enumerate}
The hidden constants depend on $\Gamma$ but are independent of $h$ and of the cut configuration.
\end{lemma}

\begin{proof}
Parametrize the connected closed polygon $\Gamma^h$ by arclength on
$[0,S_h)$. The periodic Wirtinger inequality gives~(i), and hence
\[
 \|v_h\|_{H^1(\Gamma^h)}^2
 \le\left(1+\frac{S_h^2}{4\pi^2}\right)|v_h|_{H^1(\Gamma^h)}^2
 =\left(1+\frac{S_h^2}{4\pi^2}\right)a_h(v_h,v_h),
\]
which is \eqref{eq:coercivity}. For~(ii), the same parametrization reduces
the dual problem to $-\phi_h''=f_h$ on $[0,S_h)$ with periodic boundary
conditions and zero mean, so the one-dimensional periodic elliptic estimate
applies. Under \Cref{ass:geom}, $S_h\to|\Gamma|$; hence both constants are
uniform in $h$ and the cut position.
\end{proof}

\begin{lemma}[Second-order Galerkin consistency]\label{lem:consistency}
Let $u\in H^2(\Gamma)$ be the mean-zero solution, and let $u^e$ denote
$u\circ p$ shifted to have zero mean on $\Gamma^h$. Under
\Cref{ass:geom,ass:data},
\begin{equation}\label{eq:consistency}
 \sup_{0\neq w_h\in W_h^{\mathrm{tr}}}
 \frac{|a_h(u^e,w_h)-\ell_h(w_h)|}{\|w_h\|_{H^1(\Gamma^h)}}
 \lesssim h^2\|u\|_{H^2(\Gamma)}.
\end{equation}
\end{lemma}

\begin{proof}
For $w_h\in W_h^{\rm tr}$, let $w_h^\ell=w_h\circ p^{-1}$ and set
$J_h=ds/ds_h$, viewed on $\Gamma$ after pullback. Writing the forms first on
$\Gamma^h$ and then changing variables through $p$, the arclength chain rule gives
\[
 a_h(u^e,w_h)
 =\int_{\Gamma^h}\partial_{s_h}u^e\,\partial_{s_h}w_h\,ds_h
 =\int_\Gamma J_h\,\partial_su\,\partial_sw_h^\ell\,ds,
\]
and
\[
 \ell_h(w_h)=\int_{\Gamma^h}g^e w_h\,ds_h
 =\int_\Gamma J_h^{-1}g w_h^\ell\,ds.
\]
The geometric estimates imply
$\|J_h-1\|_{L^\infty(\Gamma)}+\|J_h^{-1}-1\|_{L^\infty(\Gamma)}\lesssim h^2$.
Set $\overline w:=|\Gamma|^{-1}\int_\Gamma w_h^\ell\,ds$. Testing the weak
equation with $w_h^\ell-\overline w$ and using $\int_\Gamma g\,ds=0$ gives
\[
 \int_\Gamma \partial_su\,\partial_sw_h^\ell\,ds
 =\int_\Gamma g w_h^\ell\,ds.
\]
Consequently,
\[
 a_h(u^e,w_h)-\ell_h(w_h)
 =\int_\Gamma (J_h-1)\partial_su\,\partial_sw_h^\ell\,ds
  -\int_\Gamma (J_h^{-1}-1)g w_h^\ell\,ds.
\]
Cauchy--Schwarz and the closest-point norm equivalence
$\|w_h^\ell\|_{H^1(\Gamma)}\lesssim\|w_h\|_{H^1(\Gamma^h)}$ now yield
\[
 |a_h(u^e,w_h)-\ell_h(w_h)|
 \lesssim h^2\bigl(|u|_{H^1(\Gamma)}+\|g\|_{L^2(\Gamma)}\bigr)
 \|w_h\|_{H^1(\Gamma^h)}.
\]
Finally, $|u|_{H^1(\Gamma)}\le\|u\|_{H^2(\Gamma)}$ and
$g=-\Delta_\Gamma u$ gives
$\|g\|_{L^2(\Gamma)}\le\|u\|_{H^2(\Gamma)}$. Taking the supremum proves
\eqref{eq:consistency}.
\end{proof}

\subsection{\texorpdfstring{$H^1(\Gamma)$ and $L^2(\Gamma)$ error estimates}{Error estimates}}
\label{sec:error-estimates}

\begin{theorem}[$H^1$ error estimate]\label{thm:H1-error}
Let $u$ be the mean-zero solution with $u\in H^2(\Gamma)$, and let $u_h\in W_{h,\diamond}^{\mathrm{red}}$ be the reduced CutFEM trace solution, with lift $u^e=u\circ p$ shifted by a constant so that $\int_{\Gamma^h}u^e\,ds_h=0$. Under \Cref{ass:geom,ass:extrap-stencil,ass:data}, assume that $P_2$ is invertible for the chosen layer-potential construction. Then
\begin{equation}\label{eq:H1-error}
    \|u^e-u_h\|_{H^1(\Gamma^h)}\lesssim h\bigl(\|u\|_{H^2(\Gamma)}+\|U\|_{W^{4,\infty}(\Omega\cap U_\Gamma)}\bigr).
\end{equation}
\end{theorem}

\begin{proof}
Let $w_\star$ be the approximant of \Cref{lem:trace-approx} to $u\circ p$.
Subtract from $u\circ p$ and $w_\star$ their respective averages over
$\Gamma^h$, obtaining $u^e$ and $v_h\in W_{h,\diamond}^{\rm red}$.
Their difference is the mean-zero projection of $u\circ p-w_\star$; hence
subtracting the averages does not increase its $H^1$ norm. Therefore, with
$\mathcal N:=\|u\|_{H^2(\Gamma)}+
\|U\|_{W^{4,\infty}(\Omega\cap U_\Gamma)}$,
\[
 \|u^e-v_h\|_{H^1(\Gamma^h)}\lesssim h\mathcal N.
\]
Set $\theta_h:=v_h-u_h\in W_{h,\diamond}^{\rm red}$. Coercivity, the
Galerkin equation, continuity of $a_h$, and \Cref{lem:consistency} give
\begin{align*}
 \|\theta_h\|_{H^1(\Gamma^h)}^2
 &\lesssim a_h(\theta_h,\theta_h)\\
 &=a_h(v_h-u^e,\theta_h)
   +a_h(u^e,\theta_h)-\ell_h(\theta_h)\\
 &\lesssim\bigl(\|u^e-v_h\|_{H^1(\Gamma^h)}
       +h^2\|u\|_{H^2(\Gamma)}\bigr)
       \|\theta_h\|_{H^1(\Gamma^h)}.
\end{align*}
Dividing if $\theta_h\ne0$ and using the triangle inequality yields
\[
 \|u^e-u_h\|_{H^1(\Gamma^h)}
 \lesssim \|u^e-v_h\|_{H^1(\Gamma^h)}
          +h^2\|u\|_{H^2(\Gamma)}
 \lesssim h\mathcal N,
\]
which proves \eqref{eq:H1-error}.
\end{proof}

\begin{assumption}[Reduced-space approximation property]\label{ass:reduced-approx}
For every mean-zero $\phi\in H^2(\Gamma^h)$, where $H^2(\Gamma^h)$ is defined through the periodic arclength parametrization, there exists $\phi_\star\in W_{h,\diamond}^{\rm red}$ such that
\begin{equation}\label{eq:dual-approx}
   \|\phi-\phi_\star\|_{H^1(\Gamma^h)}\lesssim h\,\|\phi\|_{H^2(\Gamma^h)}.
\end{equation}
The hidden constant is independent of $h$ and the cut position. The
flat periodic model supports this scaling, but a cut-uniform proof for general
curved interfaces is not currently available; \eqref{eq:dual-approx} is
therefore an explicit assumption.
\end{assumption}

\begin{theorem}[$L^2$ error estimate]\label{thm:L2-error}
Under the hypotheses of \Cref{thm:H1-error} and \Cref{ass:reduced-approx},
\begin{equation}\label{eq:L2-error}
    \|u^e-u_h\|_{L^2(\Gamma^h)}\lesssim h^2\bigl(\|u\|_{H^2(\Gamma)}+\|U\|_{W^{4,\infty}(\Omega\cap U_\Gamma)}\bigr).
\end{equation}
\end{theorem}

\begin{proof}
Set $e_h:=u^e-u_h$. Since both $u^e$ and $u_h$ have zero mean on $\Gamma^h$, the dual problem with right-hand side $e_h$ is compatible: find $\phi\in H^1_{\diamond}(\Gamma^h)$ such that
\begin{equation}\label{eq:dual}
   a_h(z,\phi)=(e_h,z)_{\Gamma^h},\quad\forall z\in H^1_{\diamond}(\Gamma^h).
\end{equation}
By \Cref{lem:surface-estimates}(ii),
\begin{equation}\label{eq:dual-reg}
   \|\phi\|_{H^2(\Gamma^h)}\lesssim \|e_h\|_{L^2(\Gamma^h)}.
\end{equation}
\Cref{ass:reduced-approx} supplies an approximant $\phi_\star\in W_{h,\diamond}^{\rm red}$ with
\begin{equation}\label{eq:phi-approx}
   \|\phi-\phi_\star\|_{H^1(\Gamma^h)} \lesssim h\|\phi\|_{H^2(\Gamma^h)}.
\end{equation}
Then
\begin{align}
   \|e_h\|_{L^2(\Gamma^h)}^2
   &=a_h(e_h,\phi) \nonumber\\
   &=a_h(e_h,\phi-\phi_\star)+a_h(e_h,\phi_\star). \label{eq:L2-decomp}
\end{align}
For the first term, continuity, \Cref{thm:H1-error},
\eqref{eq:phi-approx}, and \eqref{eq:dual-reg} give
\[
   |a_h(e_h,\phi-\phi_\star)| \lesssim \|e_h\|_{H^1(\Gamma^h)} \|\phi-\phi_\star\|_{H^1(\Gamma^h)}
   \lesssim h^2\bigl(\|u\|_{H^2(\Gamma)}
   +\|U\|_{W^{4,\infty}(\Omega\cap U_\Gamma)}\bigr)
   \|e_h\|_{L^2(\Gamma^h)}.
\]
For the second term, the discrete Galerkin equation gives
\[
   a_h(e_h,\phi_\star)=a_h(u^e,\phi_\star)-\ell_h(\phi_\star),
\]
and \Cref{lem:consistency} gives
\[
   |a_h(e_h,\phi_\star)|\lesssim h^2\|u\|_{H^2(\Gamma)}\|\phi_\star\|_{H^1(\Gamma^h)}.
\]
Moreover,
\[
   \|\phi_\star\|_{H^1(\Gamma^h)}
   \le \|\phi\|_{H^1(\Gamma^h)}+
       \|\phi-\phi_\star\|_{H^1(\Gamma^h)}
   \lesssim \|\phi\|_{H^2(\Gamma^h)}
   \lesssim \|e_h\|_{L^2(\Gamma^h)}.
\]
Combining these bounds with \eqref{eq:L2-decomp} and dividing by $\|e_h\|_{L^2(\Gamma^h)}$ proves \eqref{eq:L2-error}.
\end{proof}

\section{Trace stability, bulk reconstruction, and conditioning}\label{sec:conditioning-section}
\subsection{One-sided trace stability}

Recall that $\Omega_h^\Gamma$ from \Cref{sec:cutfem} is the active band, an $O(h)$-neighbourhood of $\Gamma^h$.

\begin{lemma}[One-sided harmonic trace estimate: flat periodic model]
\label{lem:flat-harmonic-trace}
Let $\widehat u_h:=\mathcal E_hv_2\in V_h^{\rm red}$ and
$w_h:=\operatorname{Tr}_h\widehat u_h\in W_h^{\rm red}$. Consider the flat
grid-aligned model with periodic tangential indexing, a fixed number of active
normal layers, and the extrapolation stability of \Cref{prop:extrap-linear}(i).
Then
\[
        \|\widehat u_h\|_{L^2(\Omega_h^\Gamma)}^2
        \le
        C_{\rm tr} h\|w_h\|_{L^2(\Gamma^h)}^2 ,
\]
where $C_{\rm tr}$ is independent of $h$ and of the cut position.
\end{lemma}

\begin{proof}
Let $\gamma_2=\{(mh,0):m\in\mathbb Z_M\}$, where
$\mathbb Z_M:=\mathbb Z/M\mathbb Z$ is the periodic tangential index set,
and let $j\ge0$ count layers from $\gamma_2$ into the harmonic side. Write
$V_m=(v_2)_m$ and $U_{j,m}=(Ev_2)_{j,m}$. The discrete Fourier transform in
$m$ gives
\[
 \widehat U_j(\theta)=e^{-j\alpha(\theta)}\widehat V(\theta),
 \qquad \cosh\alpha(\theta)=2-\cos\theta,
\]
where $0\le\alpha(\theta)\le\operatorname{arcosh}(3)$. Let the cut line have
height $(d+\rho)h$, with $0\le d\le D$, $0\le\rho\le1$, and fixed $D$.
At the tangential grid points its trace values are
$t_m=(1-\rho)U_{d,m}+\rho U_{d+1,m}$, so
\[
 \widehat t(\theta)=m_{d,\rho}(\theta)\widehat V(\theta),
 \qquad
 m_{d,\rho}(\theta)
 =e^{-d\alpha(\theta)}\bigl(1-\rho+\rho e^{-\alpha(\theta)}\bigr).
\]
All factors are positive, and hence
\[
 m_{d,\rho}(\theta)
 \ge e^{-(D+1)\operatorname{arcosh}(3)}=:c_D>0.
\]
The one-dimensional finite element mass equivalence on the trace, Parseval's
identity, and this lower bound give
\[
 \|v_2\|_{\ell_h^2(\gamma_2)}^2
 =h\sum_m|V_m|^2
 \lesssim h\sum_m|t_m|^2
 \simeq\|w_h\|_{L^2(\Gamma^h)}^2.
\]
Moreover, $e^{-j\alpha(\theta)}\le1$, so every harmonic active layer has
$\ell_h^2$ norm bounded by that of $v_2$. The extrapolated $\gamma_3$ block is
bounded by \Cref{prop:extrap-linear}(i). Since the number of active layers is
fixed,
\[
 \|Ev_2\|_{\ell_h^2(\gamma)}^2
 \lesssim\|w_h\|_{L^2(\Gamma^h)}^2.
\]
Finally, full-cell $\mathcal Q_1$ mass equivalence yields
\[
 \|\widehat u_h\|_{L^2(\Omega_h^\Gamma)}^2
 \simeq h^2\sum_{\xi\in\gamma}|(Ev_2)_\xi|^2
 =h\|Ev_2\|_{\ell_h^2(\gamma)}^2
 \lesssim h\|w_h\|_{L^2(\Gamma^h)}^2.
\]
\end{proof}

\begin{assumption}[Curved-interface one-sided harmonic trace estimate]\label{ass:curved-trace}
For a smooth closed interface $\Gamma$, the reduced active reconstruction $\widehat u_h:=\mathcal E_hv_2$ and its trace $w_h:=\operatorname{Tr}_h\widehat u_h$ satisfy
\[
        \|\widehat u_h\|_{L^2(\Omega_h^\Gamma)}^2
        \le
        C_{\rm tr} h\|w_h\|_{L^2(\Gamma^h)}^2 ,
\]
with $C_{\rm tr}$ independent of $h$ and of the cut position. In particular, the bound holds if the curved reduced trace map is a sufficiently small perturbation, in the $L^2(\Gamma^h)$-operator norm, of the corresponding frozen flat trace map on a finitely overlapping patch covering; a proof of this localization is not claimed here. Direct numerical measurements are reported in \Cref{sec:numerics-observability}.
\end{assumption}

\begin{lemma}[Boundary-layer observability]
\label{lem:observability}
Let $\widehat u_h:=\mathcal E_hv_2\in V_h^{\rm red}$ and $w_h:=\operatorname{Tr}_h\widehat u_h$. Assume either the flat periodic setting of \Cref{lem:flat-harmonic-trace} or the curved-interface \Cref{ass:curved-trace}. Then there exists a constant $C > 0$, independent of $h$ and of the cut position, such that
\[
        \|v_2\|_{\ell_h^2(\gamma_2)} \le C\|w_h\|_{L^2(\Gamma^h)} .
\]
In particular, $\operatorname{Tr}_h\mathcal E_h:\mathcal M_h\to W_h^{\rm red}$
is an isomorphism.
\end{lemma}

\begin{proof}
Full-cell $\mathcal Q_1$ mass equivalence gives
\[
 \|\widehat u_h\|_{L^2(\Omega_h^\Gamma)}^2
 \simeq h^2\sum_{\xi\in\gamma}|(Ev_2)_\xi|^2
 =h\|Ev_2\|_{\ell_h^2(\gamma)}^2.
\]
Since $Ev_2=v_2$ on $\gamma_2\subset\gamma$, this equivalence and
\Cref{lem:flat-harmonic-trace} or \Cref{ass:curved-trace} yield
\[
 h\|v_2\|_{\ell_h^2(\gamma_2)}^2
 \le h\|Ev_2\|_{\ell_h^2(\gamma)}^2
 \lesssim\|\widehat u_h\|_{L^2(\Omega_h^\Gamma)}^2
 \lesssim h\|w_h\|_{L^2(\Gamma^h)}^2.
\]
Cancelling $h$ and taking square roots proves the estimate. If $w_h=0$,
then $v_2=0$, so the map is injective; it is onto $W_h^{\rm red}$ by
definition.
\end{proof}

\subsection{Bulk reconstruction consequence}
\label{sec:bulk-error}

\begin{assumption}[Quantitative bulk reconstruction]\label{ass:bulk-reconstruction}
Let $\Lambda_h:\mathbb R^{|\gamma_2|}\to\mathbb R^{h\mathbb Z^2}$ map a layer density to its full lattice potential, and let $I_h\Lambda_hq$ denote its bilinear interpolation restricted to $\Omega$. If $q_\star$ represents the exact harmonic boundary data, $P_2q_\star=\widetilde U^h_{\gamma_2}$, assume
\[
\|U-I_h\Lambda_hq_\star\|_{L^2(\Omega)}
\lesssim h^2\|U\|_{W^{4,\infty}(\Omega\cap U_\Gamma)}
\]
and, for every density $r$,
\[
\|I_h\Lambda_hr\|_{L^2(\Omega)}
\lesssim\|P_2r\|_{\ell_h^2(\gamma_2)}.
\]
These are the consistency and Dirichlet-data stability estimates imported from the companion LGF analysis~\citep{xia2026geometrically}.
\end{assumption}

\begin{corollary}[Bulk reconstruction error]\label{cor:bulk-error}
Let $U$ be the exact harmonic bulk solution of \eqref{eq:coupled-bulk}.
Assume that $P_2$ is invertible, and let $q$ solve \eqref{eq:reduced-density}
with the gauge
\[
 \int_{\Gamma^h}\operatorname{Tr}_h\mathcal E_h(P_2q)\,ds_h=0.
\]
Define $U_h^{\mathrm{lat}}:=\Lambda_hq$ and $U_h:=I_hU_h^{\mathrm{lat}}$.
Under \Cref{ass:geom,ass:extrap-stencil,ass:data,ass:bulk-reconstruction,ass:reduced-approx},
and either the flat periodic setting of \Cref{lem:flat-harmonic-trace} or
\Cref{ass:curved-trace},
\[
 \|U-U_h\|_{L^2(\Omega)}\lesssim h^2\bigl(\|u\|_{H^2(\Gamma)}+\|U\|_{W^{4,\infty}(\Omega\cap U_\Gamma)}\bigr),
\]
with constants independent of the cut configuration.
\end{corollary}

\begin{proof}
Set $w_h:=\operatorname{Tr}_h\mathcal E_h(P_2q)$. By
\Cref{prop:equivalence} and the imposed gauge, $w_h$ is the mean-zero reduced
trace solution. Let $q_\star$ be as in \Cref{ass:bulk-reconstruction}. Then
\[
 U-U_h=\bigl(U-I_h\Lambda_hq_\star\bigr)+I_h\Lambda_h(q_\star-q).
\]
The first term is controlled by \Cref{ass:bulk-reconstruction}. For the
approximant of \Cref{lem:trace-approx},
$v_{\star,2}=\widetilde U^h_{\gamma_2}=P_2q_\star$; hence
\Cref{prop:equivalence,lem:observability} give
\[
 w_\star-w_h
 =\operatorname{Tr}_h\mathcal E_h\bigl(P_2(q_\star-q)\bigr),
 \qquad
 \|P_2(q_\star-q)\|_{\ell_h^2(\gamma_2)}
 \lesssim\|w_\star-w_h\|_{L^2(\Gamma^h)}.
\]
Let
\[
 c_h:=|\Gamma^h|^{-1}\int_{\Gamma^h}u\circ p\,ds_h,
 \qquad u^{e,0}:=u\circ p-c_h.
\]
Since $u$ has zero mean on $\Gamma$, geometric measure consistency gives
$|c_h|\lesssim h^2\|u\|_{H^2(\Gamma)}$. Therefore
\Cref{lem:trace-approx,thm:L2-error} and the triangle inequality yield
\[
 \|w_\star-w_h\|_{L^2(\Gamma^h)}
 \le \|w_\star-u\circ p\|_{L^2(\Gamma^h)}
 +|\Gamma^h|^{1/2}|c_h|
 +\|u^{e,0}-w_h\|_{L^2(\Gamma^h)}
 \lesssim h^2\bigl(\|u\|_{H^2(\Gamma)}
 +\|U\|_{W^{4,\infty}(\Omega\cap U_\Gamma)}\bigr).
\]
The stability estimate in \Cref{ass:bulk-reconstruction} controls
$I_h\Lambda_h(q_\star-q)$ by the same quantity and completes the proof.
\end{proof}

\begin{remark}[Inhomogeneous bulk reconstruction]
For $-\Delta U=f$, \eqref{eq:w2-inhomogeneous} and
\eqref{eq:reduced-density-inhomog} define the algebraic method. Extending the
corollary also requires a stable second-order lattice approximation of a smooth
particular solution and a Strang estimate for the surface traces generated by
the affine coefficient set $u_\gamma^p+\operatorname{range}(F)$; these estimates
are not established here.
\end{remark}

\subsection{Cut-independent conditioning of \texorpdfstring{$K_{\mathrm{red}}$}{Kred}}
\label{sec:conditioning}

Observability upgrades the surface trace coercivity to an exact coefficient-space nullspace statement and supplies the lower Rayleigh bound; the upper bound follows from reconstruction stability.

\begin{proposition}[Nullspace and discrete gauge]\label{prop:gauge}
Under the hypotheses of \Cref{lem:observability}, assume in addition that
$P_2$ is invertible and that $\Gamma^h$ is connected. Then
\[
 \ker K_{\rm red}=\operatorname{span}\{\mathbf1_{\gamma_2}\}.
\]
Consequently, the reduced system is nonsingular under the mean-zero
constraint \eqref{eq:discrete-gauge}. Equivalently, if
\[
 m_i:=\int_{\Gamma^h}\operatorname{Tr}_h(\mathcal E_he_i)\,ds_h,
 \qquad
 K_{\rm red}^{\diamond}:=K_{\rm red}+\alpha mm^T,
 \qquad \alpha>0,
\]
then $K_{\rm red}^{\diamond}$ is symmetric positive definite, and for every
compatible load $r^\diamond\perp\mathbf1$ its solution satisfies $m^Tu_2=0$
and agrees with the constrained solution.
\end{proposition}

\begin{proof}
For $w_h:=\operatorname{Tr}_h(\mathcal E_hv_2)$, the definition of the reduced
matrix gives
\[
 v_2^TK_{\rm red}v_2=a_h(w_h,w_h)
 =\|\nabla_{\Gamma^h}w_h\|_{L^2(\Gamma^h)}^2.
\]
The inclusion $\mathbf1_{\gamma_2}\in\ker K_{\rm red}$ follows from
\eqref{eq:constant-mode}. Conversely, zero energy makes $w_h$ constant on the
connected curve $\Gamma^h$. Thus, for some $c\in\mathbb R$, constant
reproduction gives
$\operatorname{Tr}_h\mathcal E_h(v_2-c\mathbf1_{\gamma_2})=0$, and
\Cref{lem:observability} gives $v_2=c\mathbf1_{\gamma_2}$. This proves the
kernel identity. Moreover,
$m^T\mathbf1_{\gamma_2}=|\Gamma^h|\ne0$, so the kernel has trivial
intersection with the constraint space $m^Tv_2=0$. This proves constrained
nonsingularity.

Finally,
\[
 v_2^TK_{\rm red}^{\diamond}v_2
 =v_2^TK_{\rm red}v_2+\alpha(m^Tv_2)^2.
\]
Both terms are nonnegative. If their sum vanishes, the kernel identity gives
$v_2=c\mathbf1_{\gamma_2}$ and the second term gives $c=0$; hence the shifted
matrix is positive definite. If $K_{\rm red}^{\diamond}u_2=r^\diamond$ with
$\mathbf1^Tr^\diamond=0$, left multiplication by $\mathbf1^T$ gives
$0=\alpha|\Gamma^h|\,m^Tu_2$. Thus $m^Tu_2=0$, the rank-one term vanishes,
and the shifted solution is the constrained solution.
\end{proof}

\begin{lemma}[Upper trace bound for the reconstruction]\label{lem:E-bounded}
Under \Cref{ass:geom,ass:extrap-stencil}, assume that $P_2$ is invertible for
the chosen layer-potential construction. Then, for every
$v_2\in\mathbb R^{|\gamma_2|}$,
\begin{equation}\label{eq:E-L2H1}
   \|\operatorname{Tr}_h(\mathcal E_hv_2)\|_{L^2(\Gamma^h)}\lesssim \|v_2\|_{\ell_h^2(\gamma_2)},
   \qquad
   |\operatorname{Tr}_h(\mathcal E_hv_2)|_{H^1(\Gamma^h)}\lesssim h^{-1}\|v_2\|_{\ell_h^2(\gamma_2)} .
\end{equation}
Both constants are independent of the cut configuration.
\end{lemma}
\begin{proof}
Set $z_h:=\mathcal E_hv_2$. By \Cref{lem:H-stab,prop:extrap-linear}(i)
and full-cell $\mathcal Q_1$ mass equivalence,
\[
 \|Ev_2\|_{\ell_h^2(\gamma)}\lesssim\|v_2\|_{\ell_h^2(\gamma_2)},
 \qquad
 \|z_h\|_{L^2(\Omega_h^\Gamma)}^2
 \simeq h\|Ev_2\|_{\ell_h^2(\gamma)}^2.
\]
Summing the cut-independent cellwise trace inequality and using the bulk
inverse estimate gives~\citep{hansbo2002unfitted,burman2015cutfem}
\[
 \|\operatorname{Tr}_hz_h\|_{L^2(\Gamma^h)}^2
 \lesssim h^{-1}\|z_h\|_{L^2(\Omega_h^\Gamma)}^2
 \lesssim\|v_2\|_{\ell_h^2(\gamma_2)}^2.
\]
Applying the same cellwise trace estimate to $\nabla z_h$, absorbing its
derivative term by an inverse estimate, yields
\[
 |\operatorname{Tr}_hz_h|_{H^1(\Gamma^h)}^2
 \lesssim h^{-1}|z_h|_{H^1(\Omega_h^\Gamma)}^2
 \lesssim h^{-3}\|z_h\|_{L^2(\Omega_h^\Gamma)}^2
 \lesssim h^{-2}\|v_2\|_{\ell_h^2(\gamma_2)}^2.
\]
Taking square roots proves both bounds, uniformly in the cut configuration.
\end{proof}

\begin{theorem}[Cut-independent conditioning]\label{thm:conditioning}
Under \Cref{ass:geom,ass:extrap-stencil}, assume that $P_2$ is invertible,
$\Gamma^h$ is connected, and either the flat periodic setting of
\Cref{lem:flat-harmonic-trace} or the curved-interface
\Cref{ass:curved-trace} holds. Then the nonzero eigenvalues of $K_{\rm red}$,
measured in the Euclidean coefficient norm, satisfy
\begin{equation}\label{eq:kappa}
 \lambda_{\min}^+(K_{\rm red})\gtrsim h,
 \qquad \lambda_{\max}(K_{\rm red})\lesssim h^{-1},
 \qquad
 \kappa^+(K_{\rm red})
 :=\frac{\lambda_{\max}(K_{\rm red})}{\lambda_{\min}^+(K_{\rm red})}
 \lesssim h^{-2}.
\end{equation}
All constants are independent of $h$ and of the cut configuration.
\end{theorem}
\begin{proof}
The kernel is exactly the constant mode by \Cref{prop:gauge}.
For a coefficient vector $v_2$, set $\widehat v_h:=\mathcal E_hv_2\in V_h^{\rm red}$ and $v_h:=\operatorname{Tr}_h\widehat v_h\in W_h^{\rm red}$. The Rayleigh quotient is
\[
   v_2^TK_{\rm red}v_2=a_h(v_h,v_h)=|v_h|_{H^1(\Gamma^h)}^2.
\]
The lower bound is first established on the discrete mean-zero subspace $\{v_2:m^Tv_2=0\}$, where $m_i=\int_{\Gamma^h}\operatorname{Tr}_h(\mathcal E_he_i)\,ds_h$. On this subspace, \Cref{lem:surface-estimates}(i) and the observability bound \Cref{lem:observability} give
\[
   v_2^TK_{\rm red}v_2 \gtrsim \|v_h\|_{L^2(\Gamma^h)}^2 \gtrsim \|v_2\|_{\ell_h^2(\gamma_2)}^2 =h\|v_2\|_{\ell^2(\gamma_2)}^2.
\]
To transfer this bound to the Euclidean orthogonal complement
$\{\mathbf1\}^\perp$, take $v_2\perp\mathbf1$ and set
$w:=v_2-(m^Tv_2/m^T\mathbf1)\mathbf1$. Then $m^Tw=0$ and
$K_{\rm red}w=K_{\rm red}v_2$. Since $v_2\perp\mathbf1$,
$\|w\|_2^2\ge\|v_2\|_2^2$, and therefore
\[
 v_2^TK_{\rm red}v_2=w^TK_{\rm red}w
 \gtrsim h\|w\|_2^2\gtrsim h\|v_2\|_2^2.
\]
This proves the lower eigenvalue bound. For the upper bound,
\Cref{lem:E-bounded} gives, for every $v_2$,
\[
   v_2^TK_{\rm red}v_2=|v_h|_{H^1(\Gamma^h)}^2
   \lesssim h^{-2}\|v_2\|_{\ell_h^2(\gamma_2)}^2
   =h^{-1}\|v_2\|_{\ell^2(\gamma_2)}^2.
\]
Combining the two bounds proves \eqref{eq:kappa}.
\end{proof}

\begin{remark}[Coefficient norm and gauge]\label{rem:conditioning-norm}
The variational gauge uses $m^\perp$, whereas the nonzero spectrum of the
symmetric matrix is characterized on $\mathbf1^\perp$; the transfer in the
proof links these two complements. The numerical section reports MATLAB
\texttt{cond} of the rank-one-shifted matrix. Its scale-matched shift adds a
gauge eigenvalue comparable to the top of the nonzero spectrum, so it has the
same $h^{-2}$ asymptotic order as \eqref{eq:kappa}, although the finite values
need not agree.
\end{remark}

\subsection{Conditioning of the density formulations: flat-interface symbol analysis}
\label{sec:density_conditioning}

The result of this subsection is a fully rigorous conditioning analysis for the density formulations on the flat periodic interface, the setting in which discrete pseudodifferential calculus is elementary. The same scalings are confirmed numerically on circular and deformed interfaces in \Cref{sec:numerics}. A rigorous extension to curved interfaces would require a discrete pseudodifferential calculus for layer operators restricted to a curved active vertex set and is left for future work.

\begin{lemma}[Discrete potential and stiffness symbols on a flat interface] \label{lem:symbols}
Let the flat periodic interface have fixed tangential period $L=Mh$, and let
$\Theta_M:=\{2\pi k/M:k\in\mathbb Z_M\}$, represented in $[-\pi,\pi)$, be
the resolved frequency set. For
$\widehat v(\theta)=\sum_{m\in\mathbb Z_M}v_me^{-im\theta}$ and every
$\theta\in\Theta_M\setminus\{0\}$, the surface stiffness and the LGF trace
operators have symbols
\begin{align}
   \widehat k_h(\theta)&=\frac4h\sin^2\frac{\theta}{2},\label{eq:k-symbol}\\
   \widehat s_h(\theta)&=\frac{1}{2\sinh\alpha(\theta)}
      \simeq \frac{1}{\alpha(\theta)},
      \qquad \cosh\alpha(\theta)=2-\cos\theta,\label{eq:s-symbol}\\
   \widehat d_h(\theta)&=-\frac{1}{1+e^{\alpha(\theta)}}\simeq 1.
      \label{eq:d-symbol}
\end{align}
Here $a\simeq b$ means $c^{-1}|b|\le|a|\le c|b|$, with $c$ independent
of $h$ and of the nonzero resolved mode.
\end{lemma}
\begin{proof}
The Fourier transform of the surface stencil $h^{-1}(-1,2,-1)$ gives
\eqref{eq:k-symbol}. For the LGF, put the source layer at $j=0$ and transform
in the tangential index. Away from that layer, the normal sequence satisfies
\[
   \hat u_{j+1}-(4-2\cos\theta)\hat u_j+\hat u_{j-1}=0.
\]
The decaying solution is $\hat u_j=A(\theta)e^{-\alpha(\theta)|j|}$, with
$\cosh\alpha=2-\cos\theta$. Since the scaled point source in
$-L_hG_h=\delta_0$ has value $h^{-2}$, multiplication by $h^2$ leaves the
unit jump
\[
   A(\theta)\bigl(2e^{-\alpha}-(4-2\cos\theta)\bigr)
   = A(\theta)\bigl(e^{-\alpha}-e^{\alpha}\bigr)
   = -2A(\theta)\sinh\alpha(\theta)=-1.
\]
Thus $A=(2\sinh\alpha)^{-1}$, which is the single-layer trace symbol.
Moreover, $\sinh(\alpha/2)=|\sin(\theta/2)|$, so
$\sinh\alpha\simeq\alpha\simeq|\sin(\theta/2)|$ uniformly on the resolved
range. In the flat geometry each source has one exterior normal neighbour.
The difference in \eqref{eq:SD-def} therefore gives
\[
 \widehat d_h(\theta)
 =A(\theta)(e^{-\alpha(\theta)}-1)
 =-\frac{1}{1+e^{\alpha(\theta)}},
\]
which is uniformly bounded above and below in magnitude and extends
continuously to $\widehat d_h(0)=-1/2$. The zero mode is omitted from the
density stiffness because $\widehat k_h(0)=0$.
\end{proof}

\begin{theorem}[Flat-symbol conditioning of density formulations] \label{thm:density-conditioning}
In the fixed-period flat periodic model of
\Cref{lem:flat-harmonic-trace,lem:symbols}, the nonzero eigenvalues satisfy
\begin{equation}\label{eq:density-eigenvalues}
 \lambda_{\min}^+(\widetilde K_{\rm red}^{S})
 \simeq\lambda_{\max}(\widetilde K_{\rm red}^{S})\simeq h^{-1},
 \qquad
 \lambda_{\min}^+(\widetilde K_{\rm red}^{D})\simeq h,
 \quad
 \lambda_{\max}(\widetilde K_{\rm red}^{D})\simeq h^{-1}.
\end{equation}
Consequently,
\begin{equation}\label{eq:density-kappa}
   \kappa^+(\widetilde K_{\rm red}^{S})\lesssim 1,
   \qquad
   \kappa^+(\widetilde K_{\rm red}^{D})\lesssim h^{-2},
\end{equation}
where $S$ and $D$ denote the single- and double-layer parametrizations and
$\kappa^+$ is the condition number on the nonzero Fourier subspace. The
constants are uniform in $h$ and in the flat-interface cut offset.
\end{theorem}
\begin{proof}
All operators commute with tangential shifts, so
$\widetilde K_{\rm red}=F^TKF$ is diagonalized by the discrete Fourier modes.
For a nonzero mode $\theta$, the density reconstruction has master-layer
symbol $\widehat p_2(\theta)$, equal to $\widehat s_h(\theta)$ or
$\widehat d_h(\theta)$. By \Cref{lem:flat-harmonic-trace}, its trace on the
cut line has the additional multiplier $m_{d,\rho}(\theta)$. Hence the
corresponding eigenvalue is
\[
 |\widehat p_2(\theta)|^2|m_{d,\rho}(\theta)|^2\widehat k_h(\theta),
\]
where
$e^{-(D+1)\operatorname{arcosh}3}\le|m_{d,\rho}|\le1$ uniformly in the cut
offset. For the single layer, \Cref{lem:symbols} gives
\[
   |\widehat s_h(\theta)|^2|m_{d,\rho}(\theta)|^2\widehat k_h(\theta)
   \simeq \left(\frac{1}{\alpha(\theta)}\right)^2
   |m_{d,\rho}(\theta)|^2\frac4h\sin^2\frac{\theta}{2}
   \simeq h^{-1}
\]
on every nonzero resolved mode. For the double layer,
$|\widehat d_h(\theta)|\simeq1$, and therefore its eigenvalue is comparable
to $h^{-1}\sin^2(\theta/2)$. Since the period $L=Mh$ is fixed, the smallest
nonzero resolved frequency satisfies $|\sin(\theta/2)|\simeq h$, while the
largest is of order one. This proves \eqref{eq:density-eigenvalues} and
\eqref{eq:density-kappa}. The zero mode is excluded because
$\widehat k_h(0)=0$.
\end{proof}

\begin{remark}[Operator preconditioning and curved interfaces]\label{rem:operator-precond}
The spectral cancellation of the single-layer potential is the discrete analogue of Calder\'on preconditioning~\citep{steinbach1998construction,hiptmair2006operator}: the Laplace--Beltrami operator has order $+2$ and the single layer has order $-1$, so the congruence produces an order-zero operator. For smooth closed curves the flat calculation is the principal-symbol model obtained by localizing and flattening; curvature and $\gamma_3$ extrapolation are lower-order perturbations. A fully rigorous curved-interface proof requires a discrete pseudodifferential calculus for lattice layer operators on curved active vertex sets and is left to future work; the numerical experiments in \Cref{sec:numerics} confirm the scalings for the circle and deformed curves.
\end{remark}

\subsection{Extensions to other bulk operators}\label{sec:extensions}

The stabilization mechanism of this paper---constraining the surface trial space via a discrete bulk extension and reducing the system through a Galerkin projection---does not depend on the specific choice of the five-point Laplacian. In this subsection we discuss the framework's extension to other bulk operators, first through a concrete symbol-level analysis for the screened Poisson (modified Helmholtz) operator, and then through a broader discussion of the applicable operator class.

\subsubsection{Screened Poisson matching}

The argument below is presented at the level of the principal symbol on a flat periodic interface and is supported numerically in \Cref{sec:numerics-sweep}.

When positive reaction terms are added both in the bulk and on the surface, $-\Delta u+\sigma_{\mathrm{bulk}}u=f$ in $\Omega$ and $-\Delta_\Gamma u+\sigma_{\mathrm{surface}}u=g$ on $\Gamma$, the constant mode is no longer in the kernel and no mean-zero gauge is required. On a locally flat interface, the screened surface symbol is
\[
    k_h^{\sigma_{\mathrm{surface}}}(\theta)\simeq\tfrac{4}{h}\sin^2(\theta/2)+\sigma_{\mathrm{surface}}h
\]
and the screened single-layer symbol is $s_h^{\sigma_{\mathrm{bulk}}}(\theta)\simeq 1/\sqrt{4\sin^2(\theta/2)+\sigma_\mathrm{bulk}h^2}$. The single-layer density formulation has principal symbol
\[
    k_h^{\sigma_{\mathrm{surface}}}(\theta)|s_h^{\sigma_{\mathrm{bulk}}}(\theta)|^2\simeq h^{-1}\frac{4\sin^2(\theta/2)+\sigma_\mathrm{surface}h^2}{4\sin^2(\theta/2)+\sigma_{\mathrm{bulk}}h^2}.
\]
The screened bulk operator therefore acts as a matched preconditioner when $\sigma_\mathrm{bulk}=\sigma_\mathrm{surface}$: the principal symbol is constant ($\simeq h^{-1}$). For mismatched parameters, conditioning remains bounded in $h$ with prefactor controlled by
\[
    \max\{1,\sigma_\mathrm{bulk}/\sigma_\mathrm{surface},\sigma_\mathrm{surface}/\sigma_\mathrm{bulk}\},
\]
excluding division by zero. The cut-independent stabilization mechanism is unchanged.

\subsubsection{Broader operator class}

The reduction extends to a Cartesian lattice operator $\mathcal{L}_h$ provided
(i) an efficiently evaluable LGF $G_h$ exists;
(ii) its layer potentials have a stably invertible trace block $P_2$;
(iii) the $\mathcal{L}_h$-harmonic extension satisfies a maximum principle or
an $\ell_h^2$ stability bound analogous to \Cref{lem:H-stab}; and
(iv) that extension is second-order consistent for smooth solutions.

The screened Laplacian treated above and monotone constant-coefficient
anisotropic scalar stencils are the most direct candidates. Although Cartesian
LGFs are available for the biharmonic and Stokes operators
\citep{liska2014parallel,liska2016fast}, those operators lack the scalar maximum
principle used here and require separate layer-block and reconstruction-stability
results; they are therefore not covered. For variable coefficients, the LGF of
the principal constant-coefficient part could instead serve as a parametrix,
with an additional perturbation analysis.

On the surface, replacing $-\Delta_\Gamma$ by
$-\nabla_\Gamma\cdot(a_\Gamma\nabla_\Gamma)+c_\Gamma$, with
$0<a_0\le a_\Gamma\le a_1$ and $c_\Gamma\ge0$, preserves the reduction and
conditioning through spectral equivalence. A genuinely two-way law,
such as $\partial_nu_b=k(u_s-u_b)$ coupled to a surface balance, would instead
produce a coupled boundary system for the layer density and surface unknown;
its analysis lies beyond the present one-way framework.

\section{Numerical results}\label{sec:numerics}

We validate the theoretical results on benchmark problems designed to test optimal convergence, cut-independent conditioning, and the different mesh-size conditioning scalings of the direct and density formulations. The lattice Green's function is precomputed on a background box with machine precision. The circle, cut-translation, parameter-sweep, and trace-observability experiments use the Cartesian auxiliary box $[-1.2,1.2]^2$ with $N\times N$ interior lattice points, mesh size $h=2.4/(N+1)$, and $N=2^{k}-1$ for $k=7,8,9,10$ (so the tabulated labels $128,256,512,1024$ refer to $N+1$). The deformed-circle convergence experiment uses $[-1.5,1.5]^2$, hence $h=3/(N+1)$, at the same four values of $N+1$. The MATLAB \texttt{pcg} solver is used with tolerance $\tau=10^{-10}$, and two-point Gaussian quadrature is used on each interface segment $\Gamma_K^h$.

As an unreduced algebraic baseline, the full active-grid stiffness $K$ is singular for every cut, so its ordinary MATLAB condition number is \texttt{Inf}; the small-cut pathology appears in its nonzero spectrum as the smallest positive eigenvalue collapses. The figures below therefore report the solvable reduced systems. 

\subsection{Circle: zero mean}\label{sec:numerics-circle}
Let $\Gamma$ be the unit circle $x^2+y^2=1$ and the manufactured solution be $u(x,y)=x^2-y^2$. The exact surface PDE forcing $g$ is computed analytically; bulk harmonicity is satisfied automatically.

\Cref{fig:condition_number} presents the condition numbers of $K_{\mathrm{red}}$ in $E$ mode and of $\widetilde K_{\mathrm{red}}^S$ and $\widetilde K_{\mathrm{red}}^D$ in $F$ mode, computed with \texttt{cond} in MATLAB, after imposing the discrete gauge/rank-one shift. These are shifted-matrix condition numbers, not the nonzero-spectrum quotient used in \Cref{rem:conditioning-norm}. The condition numbers of $K_{\mathrm{red}}$ and $\widetilde K_{\mathrm{red}}^D$ grow as $O(h^{-2})$, while those of $\widetilde K_{\mathrm{red}}^S$ remain uniformly bounded under mesh refinement, consistent with our analysis. Subsequent results are obtained with $\widetilde K_{\mathrm{red}}^S$ using \texttt{pcg} with no additional preconditioning and tolerance $\tau=10^{-10}$.

\begin{figure}[htbp]
\centering
\includegraphics[width=0.5\textwidth, trim = 2cm 7cm 2cm 6.5cm]{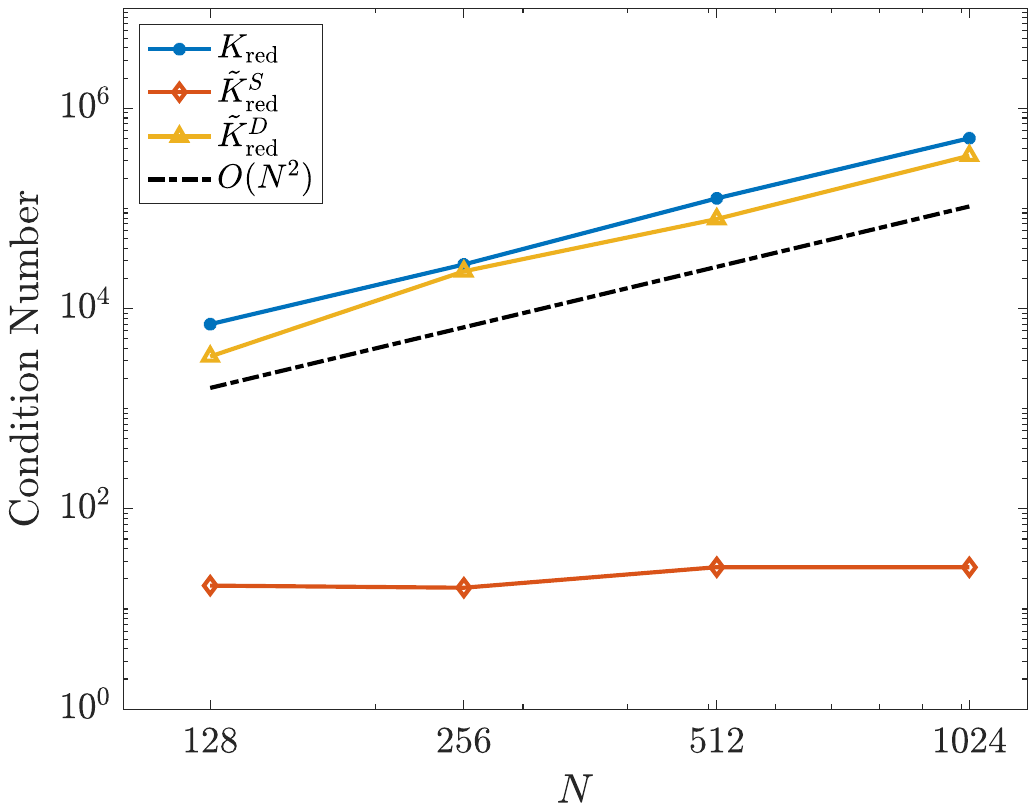}
\caption{Rank-one-shifted condition number versus $N+1=2.4/h$ for the direct $E$ formulation and the single- and double-layer $F$ formulations. The direct and double-layer curves follow the reference order $h^{-2}$; the single-layer curve remains bounded.}\label{fig:condition_number}
\end{figure}

\Cref{tab:circle} reports iteration counts of \texttt{pcg}, condition numbers, and the observed second-order convergence in the bulk $L^2$ norm and surface $L^2$ norm, with first-order convergence in the surface $H^1$ norm. The increase from $31$ to $48$ iterations across the refinement levels does not contradict bounded conditioning: \texttt{pcg} iteration counts depend on the distribution of the entire spectrum and on the components of the right-hand side, not only on the ratio of the extremal eigenvalues. We do not infer improved clustering from these counts alone.

\begin{table}[htbp]
\centering
\caption{Iteration counts, condition numbers, errors and convergence rates: bulk $L^2$, surface $L^2$, and surface $H^1$ norms.\label{tab:circle}}
\begin{tabular}{rcccccccc}
\toprule
$N{+}1$ & Iter. & Cond. &
$\|e_b\|_{L^2(\Omega)}$ & Rate &
$\|e_s\|_{L^2(\Gamma^h)}$ & Rate &
$\|e_s\|_{H^1(\Gamma^h)}$ & Rate \\
\midrule
$128$ & 31 & 17.12 & $1.2173{\times}10^{-4}$ & --   & $3.1481{\times}10^{-4}$ & --   & $2.6896{\times}10^{-2}$ & --   \\
$256$ & 37 & 16.31 & $3.1186{\times}10^{-5}$ & 1.96 & $7.8475{\times}10^{-5}$ & 2.00 & $1.3658{\times}10^{-2}$ & 0.98 \\
$512$ & 44 & 26.16 & $7.5461{\times}10^{-6}$ & 2.05 & $1.9585{\times}10^{-5}$ & 2.00 & $6.7476{\times}10^{-3}$ & 1.02 \\
$1024$ & 48 & 26.10 & $1.9180{\times}10^{-6}$ & 1.98 & $4.8991{\times}10^{-6}$ & 2.00 & $3.3365{\times}10^{-3}$ & 1.02 \\
\botrule
\end{tabular}
\end{table}

Bulk and surface errors on the $1024\times1024$ mesh and the exponentially decaying relative residuals of \texttt{pcg} are shown in \Cref{fig:circle_errors}.

\begin{figure}[htbp]
    \centering
    \subfloat[bulk]{\includegraphics[width=0.32\textwidth, trim=2cm 7cm 2cm 7cm]{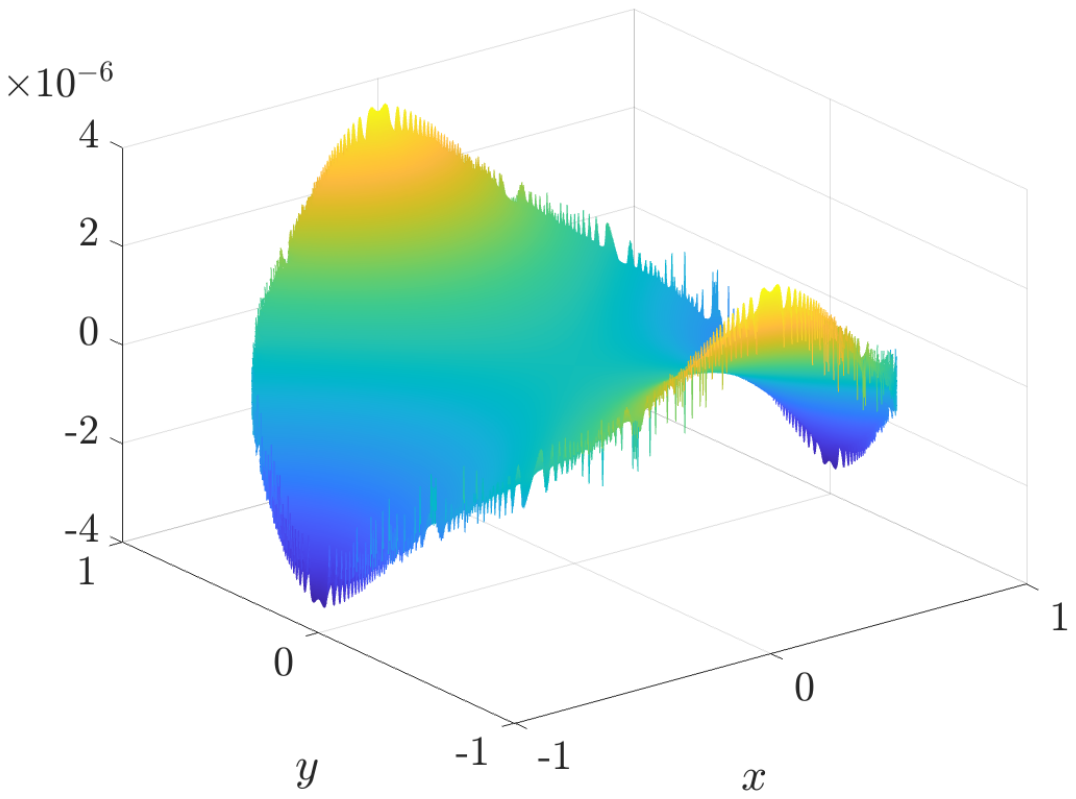}}
    \hfill
    \subfloat[surface]{\includegraphics[width=0.32\textwidth, trim=2cm 7cm 2cm 7cm]{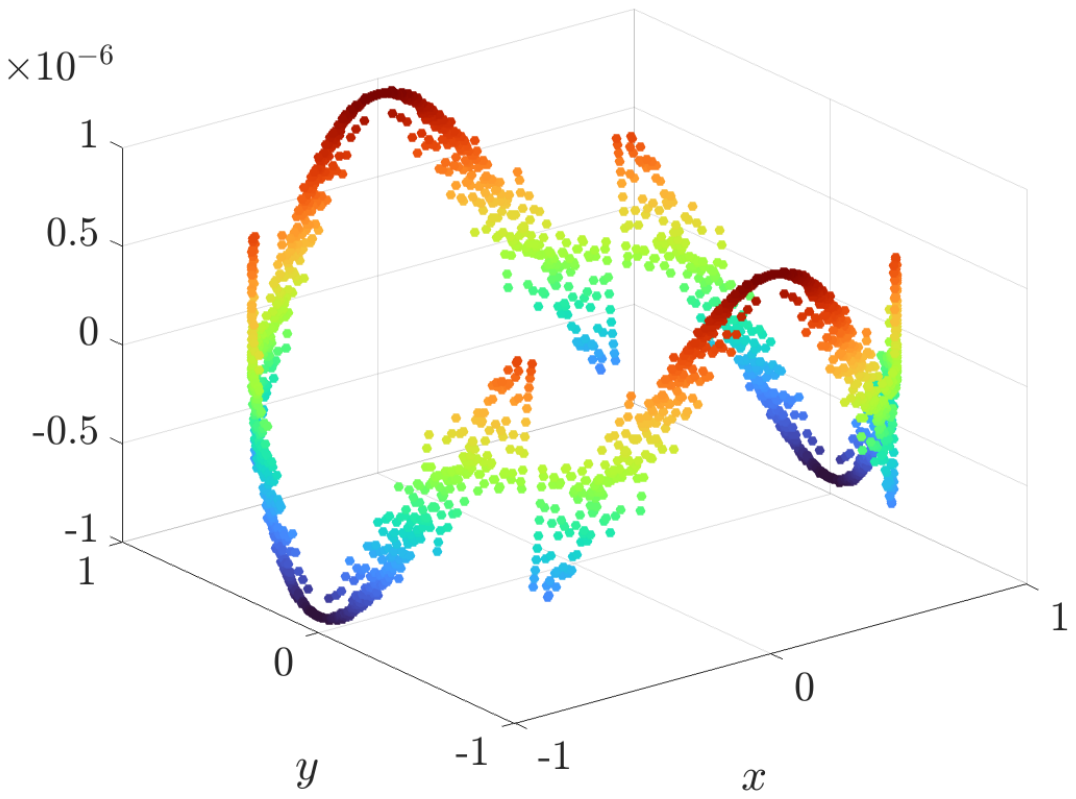}}
    \hfill
    \subfloat[relative residuals]{\includegraphics[width=0.32\textwidth, trim=2cm 7cm 2cm 7cm]{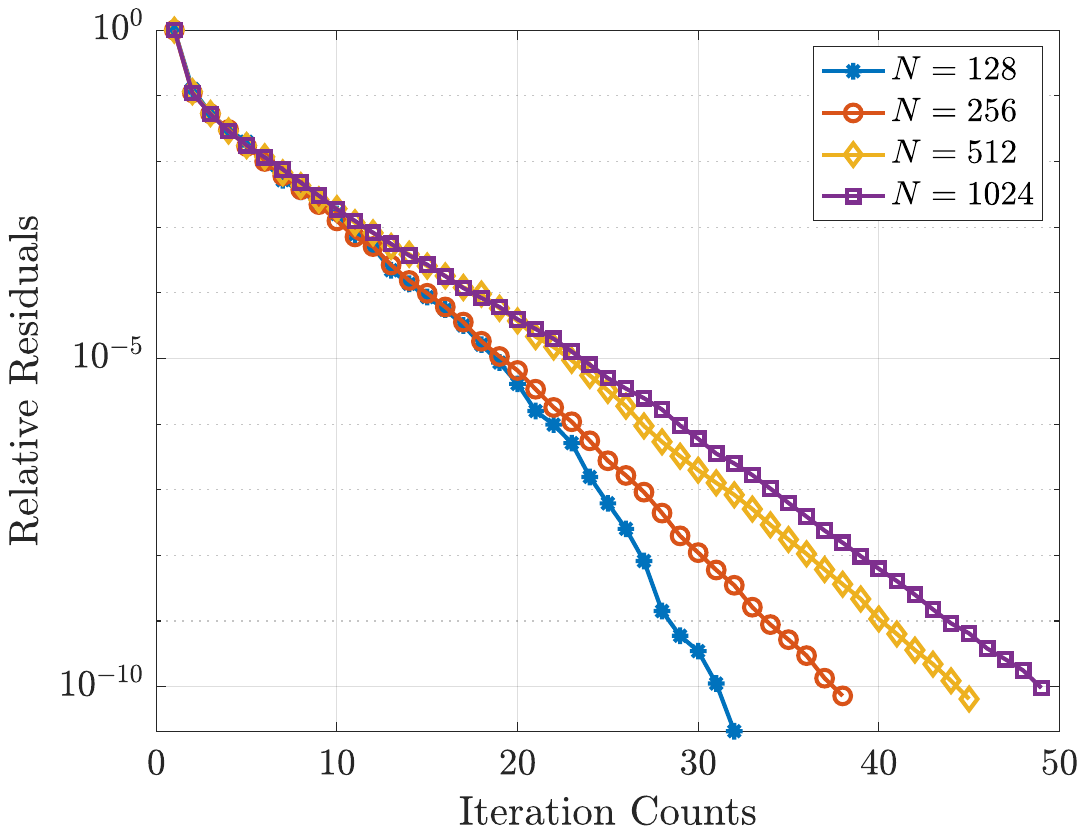}}
    \caption{Errors for the circle ($N+1=1024$) and relative residuals of the iterative solver.}
    \label{fig:circle_errors}
\end{figure}

\begin{remark}[Rank-one gauge fixing]\label{rem:gauge-alpha}
Let $(w_\gamma)_i:=\int_{\Gamma^h}\phi_i\,ds_h$ be the active-grid mean vector. For the $F$-formulation with $u_\gamma=Fq$, mean functional $m_F:=F^Tw_\gamma$, the gauge-fixed homogeneous system is
\[
(F^TKF+\alpha\,m_Fm_F^T)q=F^Tb
\] 
with $\alpha=\operatorname{tr}(K_{\rm red})/(n\,m_F^Tm_F)$, making $\operatorname{tr}(\alpha m_Fm_F^T)=\operatorname{tr}(K_{\rm red})/n$ a scale-matched gauge fixing. For an inhomogeneous bulk source $u_\gamma=u^p_\gamma+Fq$, the gauge becomes $m_F^Tq=-w_\gamma^Tu^p_\gamma$, giving $(F^TKF+\alpha\,m_Fm_F^T)q=F^Tb-F^TKu^p_\gamma-\alpha\,m_F\,w_\gamma^Tu^p_\gamma$.
\end{remark}

\subsection{Cut-position and reaction-parameter sweeps}\label{sec:numerics-sweep}
We perform two independent conditioning tests for the single-layer density
formulation: the first varies the cut position at zero reaction, while the
second fixes the geometry and varies the bulk and surface reaction parameters.

\subsubsection{Cut-translation sweep}
Set $\sigma_{\mathrm{bulk}}=\sigma_{\mathrm{surface}}=0$ and center the unit
circle at $(\beta h,0)$. We sample $101$ values of $\beta\in[-1,1]$ on the
four meshes $N+1=128,256,512,1024$. This interval contains two grid periods
and therefore samples all relative cut positions, including near-degenerate
ones. The left panel of \Cref{fig:parameter_sweep} shows that the gauge-fixed
condition numbers oscillate as the cut configuration changes but remain below
$60$ on every mesh, with no systematic growth under refinement. This is
consistent with the cut-uniformity of \Cref{thm:conditioning}.

\subsubsection{Reaction-parameter sweep}
Fix the centered circle ($\beta=0$) and the mesh $N+1=256$, and consider
\[
 -\Delta u+\sigma_{\mathrm{bulk}}u=f\text{ in }\Omega,\qquad
 -\Delta_\Gamma u+\sigma_{\mathrm{surface}}u=g\text{ on }\Gamma.
\]
We sample a $101\times101$ grid over
$(\sigma_{\mathrm{bulk}},\sigma_{\mathrm{surface}})\in[0,20]^2$. The right
panel of \Cref{fig:parameter_sweep} shows that the condition number is smallest
near the matching line
$\sigma_{\mathrm{bulk}}=\sigma_{\mathrm{surface}}$ and deteriorates as the
parameters become strongly mismatched, in agreement with the flat-symbol
calculation in \Cref{sec:extensions}. 
\begin{figure}[htbp]
    \centering
    \subfloat[cut sweep: $\sigma_{\mathrm{bulk}}=\sigma_{\mathrm{surface}}=0$]{\includegraphics[width=0.45\textwidth, trim=2cm 7cm 2cm 7cm]{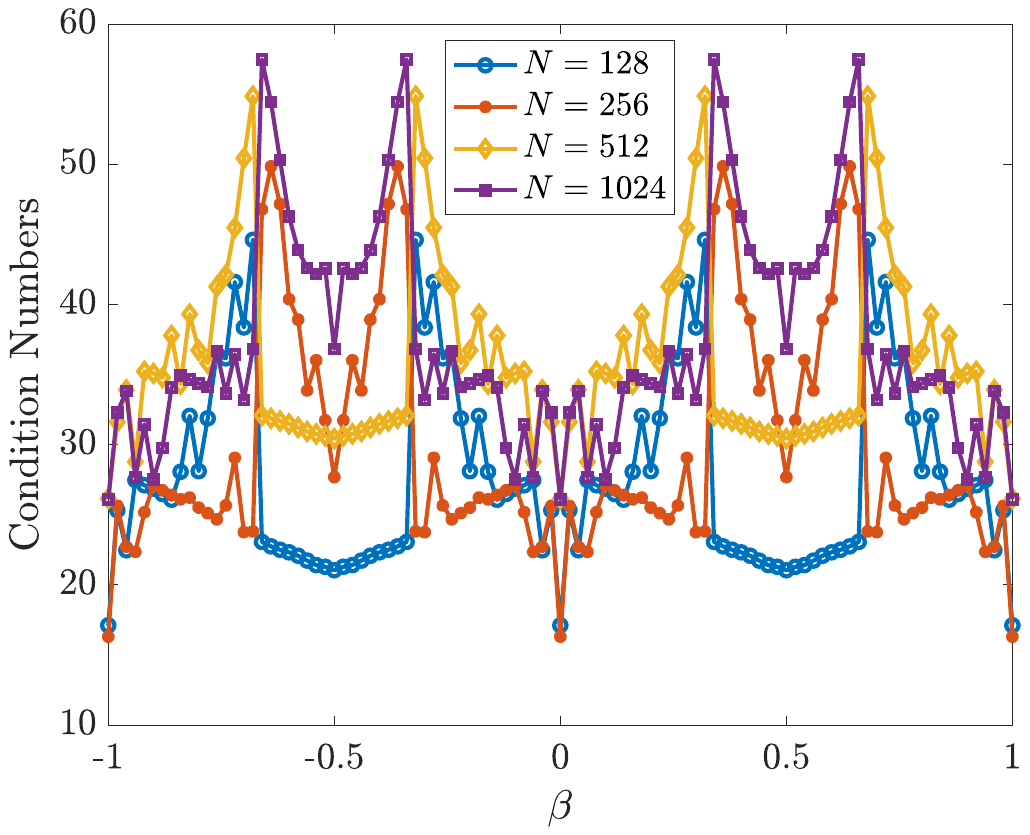}}
    \quad
    \subfloat[reaction sweep: $\beta=0$, $N+1=256$]{\includegraphics[width=0.45\textwidth, trim=2cm 7cm 2cm 7cm]{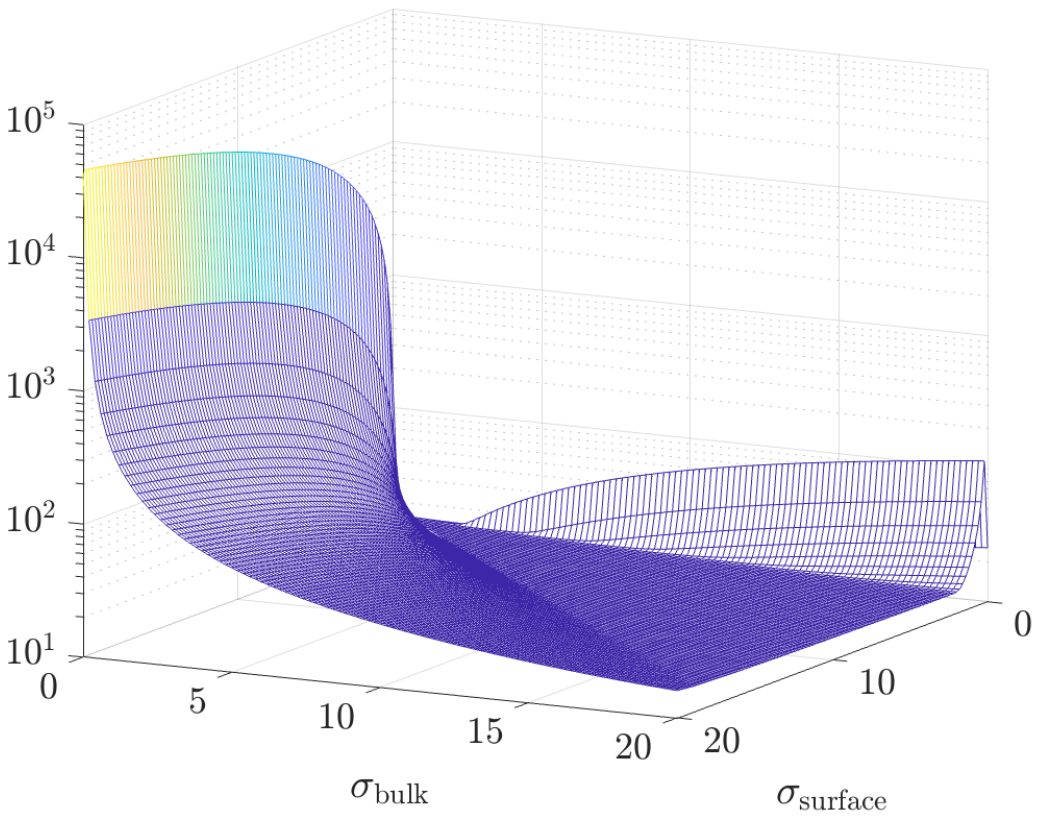}}
    \caption{Condition numbers of the single-layer density formulation
    in two independent tests. Left: the rank-one-gauge-fixed condition number
    versus the cut-translation parameter $\beta$ on four meshes at zero
    reaction. Right: the condition number over
    $(\sigma_{\rm bulk},\sigma_{\rm surface})\in[0,20]^2$ for the fixed
    centered circle, with the smallest values near the matching diagonal.}
    \label{fig:parameter_sweep}
\end{figure}

\subsection{\texorpdfstring{Direct numerical evidence for the one-sided trace estimate}{Direct numerical evidence for the one-sided trace estimate}}
\label{sec:numerics-observability}

Condition numbers alone are only indirect evidence for \Cref{ass:curved-trace}. We therefore compute the Rayleigh constants appearing in that assumption and in \Cref{lem:observability} directly. On the unit circle translated by $(\beta h,0)$ with $-1\le\beta\le 1$, using the single-layer reduction map $E$, we evaluate
\begin{align}
    C_{\mathrm{tr}}(h,\beta)
    &:=
    \sup_{v_2\neq0}
    \frac{\|u_h\|_{L^2(\Omega_h^\Gamma)}^2}{h\|u_h\|_{L^2(\Gamma^h)}^2},
    \label{eq:C-tr-num}\\
    C_{\mathrm{obs}}(h,\beta)
    &:=
    \sup_{v_2\neq0}
    \frac{\|v_2\|_{\ell_h^2(\gamma_2)}}{\|u_h\|_{L^2(\Gamma^h)}},
    \label{eq:C-obs-num}\\
    C_{\mathrm{up}}(h,\beta)
    &:=
    \sup_{v_2\neq0}
    \frac{\|u_h\|_{L^2(\Gamma^h)}}{\|v_2\|_{\ell_h^2(\gamma_2)}},
    \label{eq:C-up-num}
\end{align}
where $u_h=(Ev_2)_h$, $\Omega_h^\Gamma$ is the active band, and the surface and band norms are realized by the corresponding mass matrices assembled on $\Gamma^h$ and on $\Omega_h^\Gamma$. Equivalently, $C_{\mathrm{tr}}$ is obtained from the largest generalized eigenvalue of the band/surface mass pair (scaled by $h^{-1}$), while $C_{\mathrm{obs}}$ and $C_{\mathrm{up}}$ are obtained from the extremal generalized eigenvalues of the surface/$\ell_h^2$ pair. The same meshes as in \Cref{sec:numerics-sweep} are used ($N+1=128,256,512,1024$), with $101$ values of $\beta$ equally spaced in $[-1,1]$.

\begin{figure}[htbp]
    \centering
    \subfloat[$C_{\mathrm{tr}}$]{\includegraphics[width=0.32\textwidth, trim=2cm 7cm 2cm 7cm]{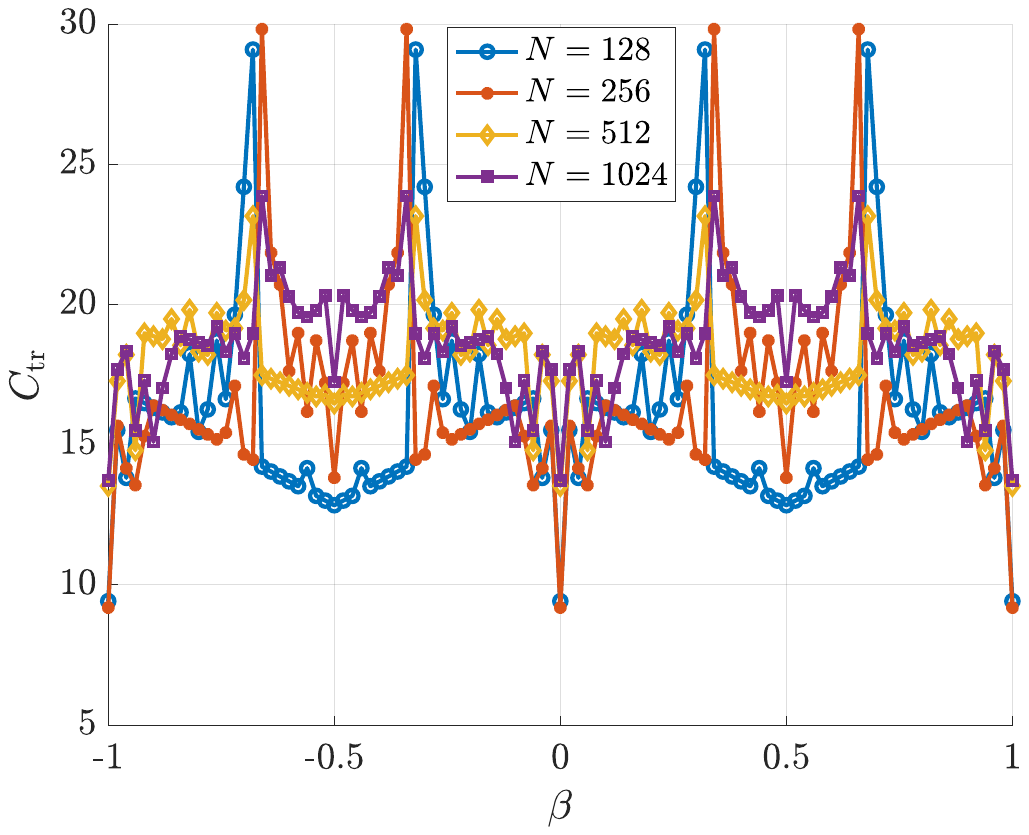}}
    \hfill
    \subfloat[$C_{\mathrm{obs}}$]{\includegraphics[width=0.32\textwidth, trim=2cm 7cm 2cm 7cm]{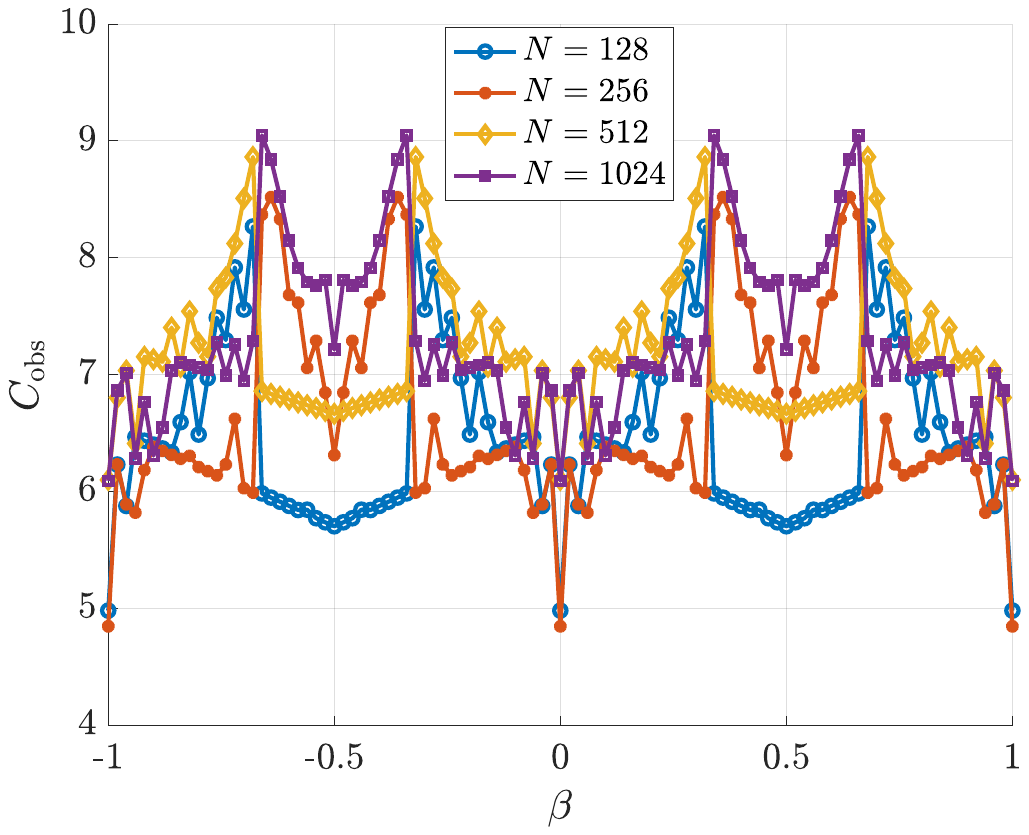}}
    \hfill
    \subfloat[$C_{\mathrm{up}}$]{\includegraphics[width=0.32\textwidth, trim=2cm 7cm 2cm 7cm]{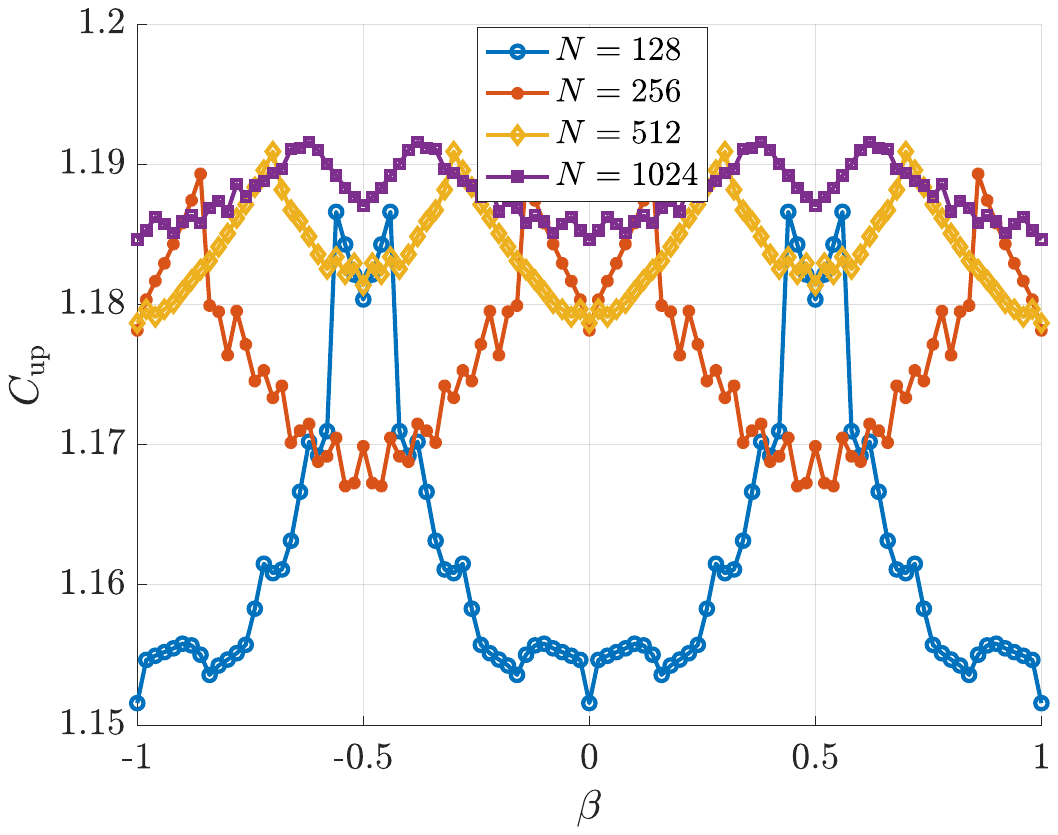}}
    \caption{Rayleigh constants \eqref{eq:C-tr-num}--\eqref{eq:C-up-num} for the reduced map $E$ on the unit circle under a one-cell translation sweep. All three quantities remain $O(1)$ in $h$ and $\beta$.}
    \label{fig:observability_sweep}
\end{figure}

\Cref{fig:observability_sweep} plots the three constants against $\beta$. All curves remain of order one: they oscillate with the cut position (as expected from changing cut ratios) but show no mesh-refinement blow-up over the tested range. \Cref{tab:observability} records the maxima over $\beta$. The trace constant satisfies $\max_\beta C_{\mathrm{tr}}\in[23,30]$ on all four meshes; the observability constant satisfies $\max_\beta C_{\mathrm{obs}}\approx 8$--$9$, increasing by about $9\%$ under an eightfold refinement, so these four levels do not by themselves exclude very slow growth; the upper trace constant $C_{\mathrm{up}}$ is essentially mesh-independent and tightly clustered near $1.19$. These measurements provide direct numerical support for \Cref{ass:curved-trace}, the observability bound used in \Cref{thm:conditioning}, and the upper reconstruction bound of \Cref{lem:E-bounded}, beyond the indirect evidence of bounded condition numbers in \Cref{fig:parameter_sweep}.

\begin{table}[htbp]
\centering
\caption{Maxima over $\beta\in[-1,1]$ of the Rayleigh constants \eqref{eq:C-tr-num}--\eqref{eq:C-up-num}.\label{tab:observability}}
\begin{tabular}{rccc}
\toprule
$N{+}1$ & $\max_\beta C_{\mathrm{tr}}$ & $\max_\beta C_{\mathrm{obs}}$ & $\max_\beta C_{\mathrm{up}}$ \\
\midrule
$128$  & $29.093$ & $8.2679$ & $1.1866$ \\
$256$  & $29.820$ & $8.5193$ & $1.1893$ \\
$512$  & $23.155$ & $8.8632$ & $1.1909$ \\
$1024$ & $23.872$ & $9.0477$ & $1.1916$ \\
\botrule
\end{tabular}
\end{table}

\subsection{Direct numerical evidence for reduced-space approximation}
\label{sec:numerics-approximation}

We next test the $h$-scaling in \Cref{ass:reduced-approx} independently of the manufactured-solution errors. Let $Z_h$ be the periodic piecewise-linear comparison space on the ordered vertices of $\Gamma^h$, and let $M_1$ and $K_1$ denote its surface mass and stiffness matrices. With the lumped mass matrix $M_L$, define the discrete comparison norm
\begin{equation}
 \|\varphi_h\|_{H^2_h(\Gamma^h)}^2
 :=\boldsymbol\varphi^T
 \bigl(M_1+K_1+K_1M_L^{-1}K_1\bigr)\boldsymbol\varphi .
 \label{eq:discrete-H2-norm}
\end{equation}
The measured constant is
\begin{equation}
 C_{\mathrm{apx}}(h)
 :=\sup_{0\ne\varphi_h\in Z_h}
 \inf_{w_h\in W_h^{\mathrm{red}}}
 \frac{\|\varphi_h-w_h\|_{H^1(\Gamma^h)}}
      {h\|\varphi_h\|_{H^2_h(\Gamma^h)}} .
 \label{eq:C-apx-num}
\end{equation}
All inner products, including the cross Gram matrix between $Z_h$ and $W_h^{\mathrm{red}}$, are assembled by Gaussian quadrature on the interface segments. Eliminating the best-approximation coefficients gives the $H^1$ Schur-complement error matrix, and $C_{\mathrm{apx}}(h)^2$ is its largest generalized eigenvalue relative to $h^2$ times the matrix in \eqref{eq:discrete-H2-norm}. Constants are reproduced exactly by the reduced map, so inclusion of the constant mode does not affect the largest eigenvalue up to roundoff.

\Cref{tab:approximation-constant} reports four refinements. For the circle, the table records the maximum over the five translations $\beta\in\{-1,-\tfrac12,0,\tfrac12,1\}$ in the box $[-1.2,1.2]^2$. The centered deformed curve from \Cref{sec:numerics-deformed} uses $[-1.5,1.5]^2$. The smooth nonconvex test is the four-lobed curve $\rho=1+0.15\cos(4\theta)$ in $[-1.35,1.35]^2$. In every case $N+1$ is the tabulated resolution and $h$ is the corresponding box width divided by $N+1$.

\begin{table}[htbp]
\centering
\caption{Discrete reduced-space approximation constants \eqref{eq:C-apx-num}. The circle column is the maximum over five sub-cell translations.\label{tab:approximation-constant}}
\begin{tabular}{rccc}
\toprule
$N{+}1$ & circle, $\max_\beta$ & deformed & smooth nonconvex \\
\midrule
$64$  & $0.5934$ & $0.6674$ & $0.5843$ \\
$128$ & $0.5718$ & $0.6810$ & $0.6859$ \\
$256$ & $0.6016$ & $0.6901$ & $0.6695$ \\
$512$ & $0.6387$ & $0.6860$ & $0.6604$ \\
\botrule
\end{tabular}
\end{table}

The measured constants remain below $0.70$ and show no mesh-refinement blow-up over these levels. This is direct finite-dimensional evidence for the uniform $h$-scaling assumed in \Cref{ass:reduced-approx}, including beyond a translated circle. It is not a proof of the continuous curved-interface assumption: \eqref{eq:C-apx-num} samples the comparison space $W_h$, uses the discrete norm \eqref{eq:discrete-H2-norm}, and covers only the reported meshes and geometries.

\subsection{Deformed circle: nonzero mean gauge and inhomogeneity}\label{sec:numerics-deformed}
We deform the circle by
\[
    \rho=1+0.16\cos(2\theta+0.4)+0.1\sin(3\theta-0.7)+0.07\cos(5\theta+1.3)+0.05\sin(8\theta+0.2)
\]
and choose the manufactured solution $u(x,y)=\sin(\pi x)\cos(2\pi y)+0.25\cos(2x+y)+0.15xy+0.1x$, which gives nonzero mean gauge and nonzero right-hand sides both in the bulk and on the surface, for $\sigma_{\mathrm{bulk}}=\sigma_{\mathrm{surface}}=0$. The exact nonzero mean is used to gauge the solution in the code. \Cref{tab:deformed} shows similar second-order bulk and surface $L^2$ convergence and first-order surface $H^1$ convergence; the condition numbers remain below $100$ on all meshes.

\begin{table}[htbp]
\centering
\caption{Deformed circle: iteration counts, condition numbers, errors and convergence.\label{tab:deformed}}
\begin{tabular}{rcccccccc}
\toprule
$N$ & Iter. & Cond. &
$\|e_b\|_{L^2(\Omega)}$ & Rate &
$\|e_s\|_{L^2(\Gamma^h)}$ & Rate &
$\|e_s\|_{H^1(\Gamma^h)}$ & Rate \\
\midrule
$127$ & 49 & 28.14 & $2.2604{\times}10^{-3}$ & --   & $3.2441{\times}10^{-3}$ & --   & $1.5290{\times}10^{-1}$ & --   \\
$255$ & 56 & 36.87 & $5.6947{\times}10^{-4}$ & 1.99 & $8.1397{\times}10^{-4}$ & 1.99 & $7.6282{\times}10^{-2}$ & 1.00 \\
$511$ & 61 & 44.48 & $1.4087{\times}10^{-4}$ & 2.02 & $2.0156{\times}10^{-4}$ & 2.01 & $3.8054{\times}10^{-2}$ & 1.00 \\
$1023$ & 62 & 43.99 & $3.5661{\times}10^{-5}$ & 1.98 & $5.1714{\times}10^{-5}$ & 1.96 & $1.9187{\times}10^{-2}$ & 0.99 \\
\botrule
\end{tabular}
\end{table}

\begin{figure}[htbp]
    \centering
    \subfloat[bulk]{\includegraphics[width=0.32\textwidth, trim=2cm 7cm 2cm 7cm]{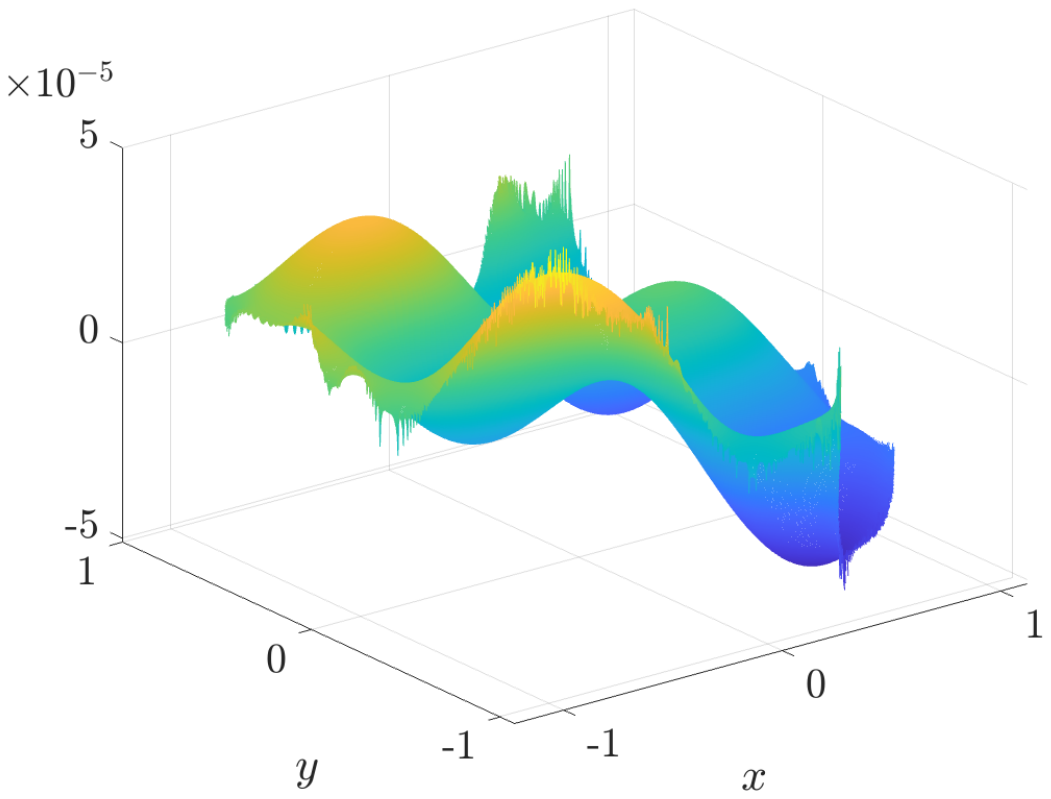}}
    \hfill
    \subfloat[surface]{\includegraphics[width=0.32\textwidth, trim=2cm 7cm 2cm 7cm]{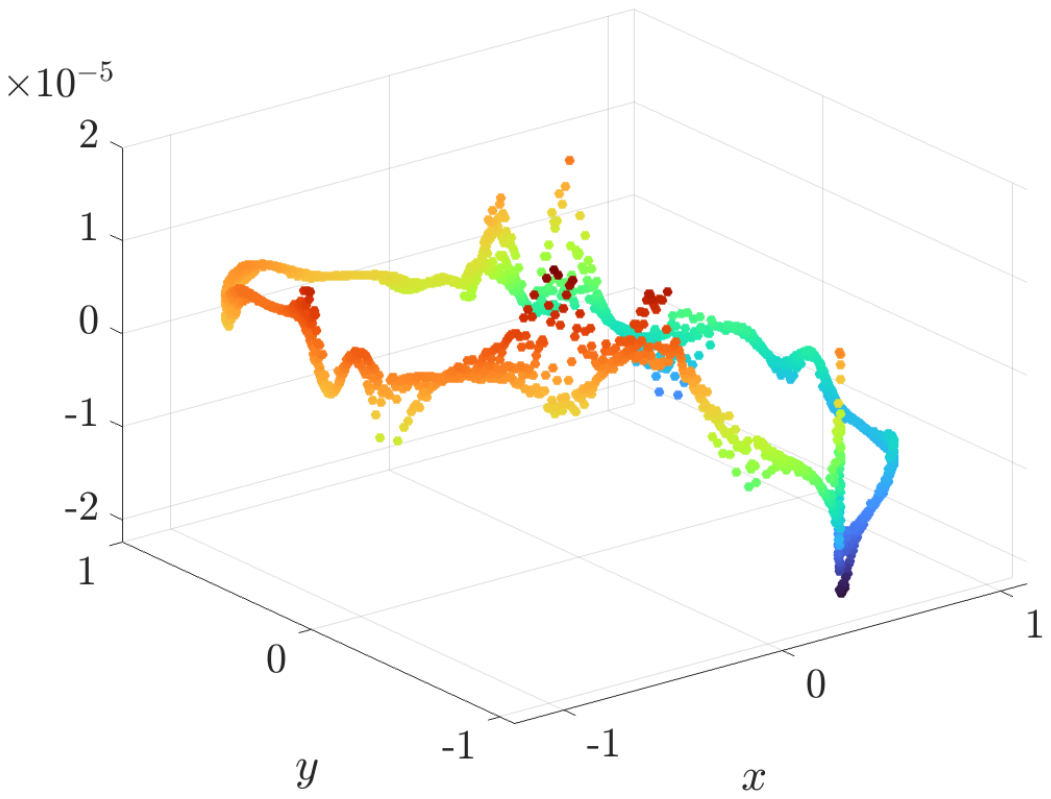}}
    \hfill
    \subfloat[relative residuals]{\includegraphics[width=0.32\textwidth, trim=2cm 7cm 2cm 7cm]{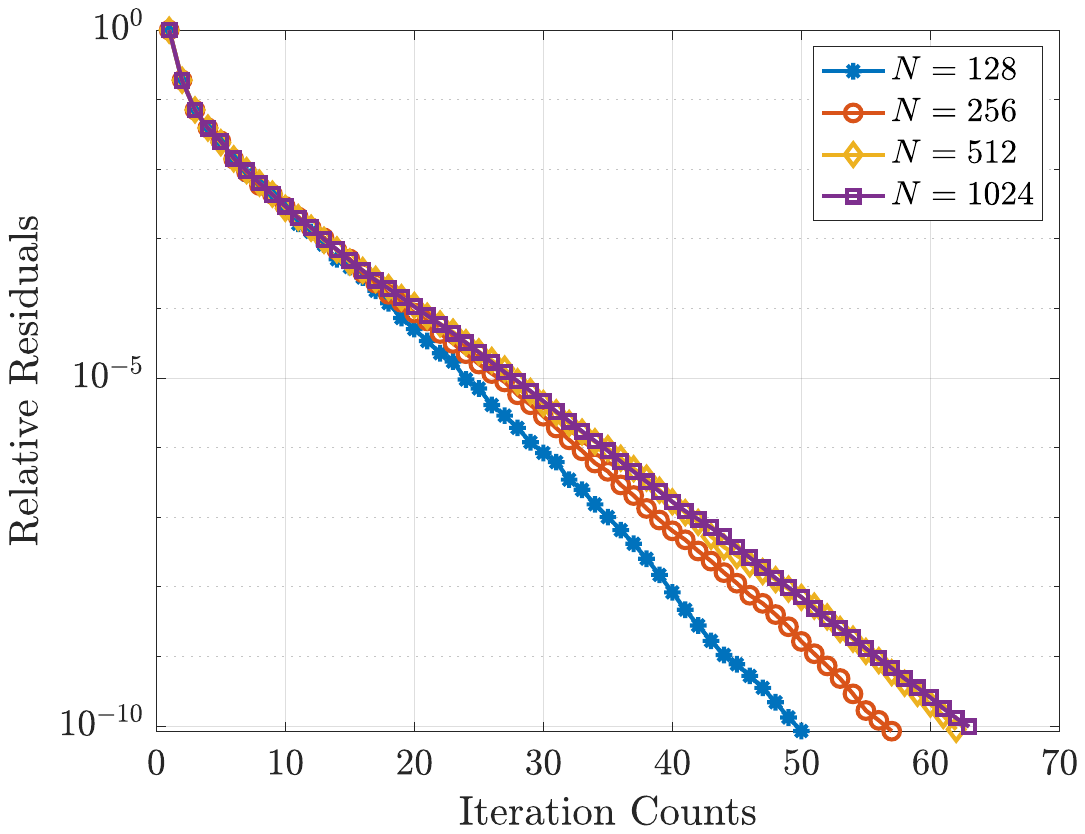}}
    \caption{Errors for the deformed circle ($N+1=1024$) and relative residuals of the iterative solver.}
    \label{fig:deformed}
\end{figure}

\Cref{fig:deformed} presents the bulk and surface errors on the $1024\times1024$ mesh and the relative residuals of the \texttt{pcg} iterative solve.

\subsection{Three-dimensional torus}\label{sec:numerics-torus-3d}

We finally test the three-dimensional FD7/extrapolation implementation on
the solid torus
\begin{equation}
 \Omega=\left\{(x,y,z)\in\mathbb R^3:
 \bigl(\sqrt{x^2+y^2}-R\bigr)^2+z^2<r^2\right\},
 \qquad \Gamma=\partial\Omega,
 \label{eq:torus-3d}
\end{equation}
centered at the origin, with major radius $R=0.8$ and minor radius
$r=0.3$. This experiment is a computational demonstration of the
three-dimensional construction; the error and conditioning analysis in
the preceding sections is two-dimensional and is not asserted here as a
three-dimensional theorem.

The manufactured solution is
\begin{equation}
 u(x,y,z)=\sin(\pi x)\cos(2\pi y)\sin(3\pi z).
 \label{eq:torus-trig2}
\end{equation}
We take $k=1$ and
$\sigma_{\mathrm{bulk}}=\sigma_{\mathrm{surface}}=0$. Thus, in the
inhomogeneous-bulk extension used by the code, $k\Delta u-
\sigma_{\mathrm{bulk}}u=f$, the bulk forcing is
$f=\Delta u=-14\pi^2u$, while the surface forcing
$g=-\Delta_\Gamma u$ is evaluated analytically from the ambient gradient
and Hessian of \eqref{eq:torus-trig2} and the exact torus normal and
curvature. The compatible surface load is projected to zero discrete
mean, and the constant mode is fixed by the same scale-matched rank-one
gauge as in \Cref{rem:gauge-alpha}, using the exact manufactured mean.

The code starts from $[-1,1]^3$ with padding $l=0.25$. For this torus the
geometry bound is $R+r=1.1$, so the padding rule gives
$l=\max\{0.25,R+r-1+0.25\}=0.35$. Thus all three runs use
\begin{equation}
 [-1.35,1.35]^3, \qquad N=2^\ell-1, \qquad
 h=\frac{2.7}{N+1}=\frac{2.7}{2^\ell},
 \label{eq:torus-box}
\end{equation}
where the input levels $\ell=5,6,7$ correspond to
$N=31,63,127$ interior points in each coordinate direction. The bulk
operator is the seven-point finite-difference Laplacian. The surface
equation is assembled with trilinear $\mathcal Q_1$ traces on the
polyhedral cut surface; interior trace values are supplied by the
three-dimensional lattice single-layer potential and the remaining
exterior cut-cell corners by second-order extrapolation.

All runs use the density map $F$ with the single-layer potential. No
additional diagonal or algebraic preconditioner is applied: the
single-layer map itself supplies the operator preconditioning. The
reduced matrix $K_F=F^TKF$ is assembled and stored, and the linear system
is solved by \texttt{pcg} with relative tolerance $10^{-10}$ and a
maximum of the reduced-system dimension. The reported condition number
is $\operatorname{cond}_2(K_F)$, where $K_F$ includes the rank-one gauge.

The results in \Cref{tab:torus-3d-errors} show approximately second-order
bulk and surface $L^2$ convergence and first-order surface $H^1$ convergence.
The trace-block condition numbers are $88.4$, $174.9$, and $353.0$, whereas
the gauge-fixed reduced condition numbers are $113.6$, $303.7$, and $249.2$;
the corresponding \texttt{pcg} counts are $79$, $137$, and $121$. Thus the
reduced conditioning and iteration counts remain of order $10^2$ over the
three levels, although neither is monotone. This provides computational
evidence, rather than a three-dimensional theorem, for the
operator-preconditioning effect of the single-layer density coordinates on a
curved surface of nontrivial topology.

\begin{table}[htbp]
\centering
\caption{Three-dimensional torus: condition numbers, errors, and observed
rates. The reduced condition number includes the rank-one mean gauge.}
\label{tab:torus-3d-errors}
\begin{tabular}{rcccccccc}
\toprule
$N$ & Iter. & Cond. &
$\|e_b\|_{L^2(\Omega)}$ & Rate &
$\|e_s\|_{L^2(\Gamma^h)}$ & Rate &
$\|e_s\|_{H^1(\Gamma^h)}$ & Rate \\
\midrule
$31$ & 79 & 113.59 &
$2.838{\times}10^{-2}$ & -- & $4.045{\times}10^{-2}$ & -- & $1.755$ & -- \\
$63$ & 137 & 303.67 &
$6.416{\times}10^{-3}$ & $2.15$ & $9.994{\times}10^{-3}$ & $2.02$ & $8.592{\times}10^{-1}$ & $1.03$ \\
$127$ & 121 & 249.23 &
$1.638{\times}10^{-3}$ & $1.97$ & $2.618{\times}10^{-3}$ & $1.93$ & $4.393{\times}10^{-1}$ & $0.97$ \\
\botrule
\end{tabular}
\end{table}


\begin{figure}[htbp]
    \centering
    \subfloat[bulk]{\includegraphics[width=0.32\textwidth]{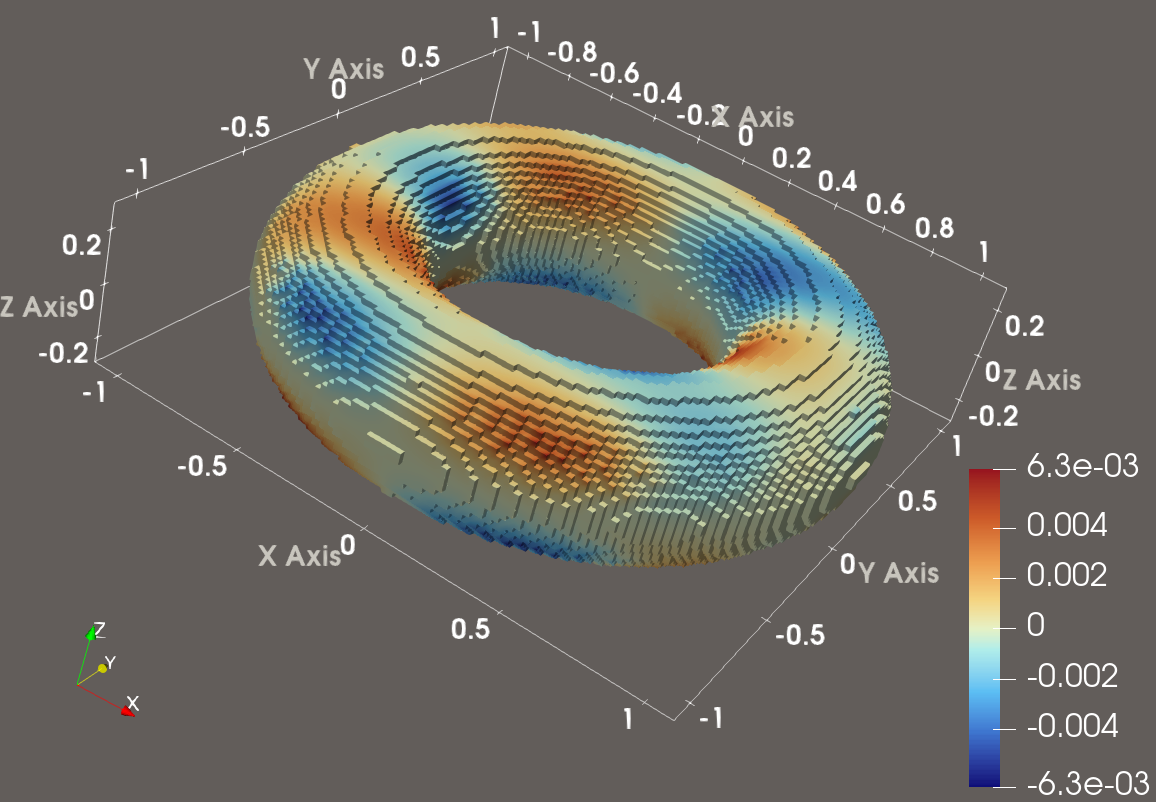}}
    \hfill
    \subfloat[surface]{\includegraphics[width=0.32\textwidth]{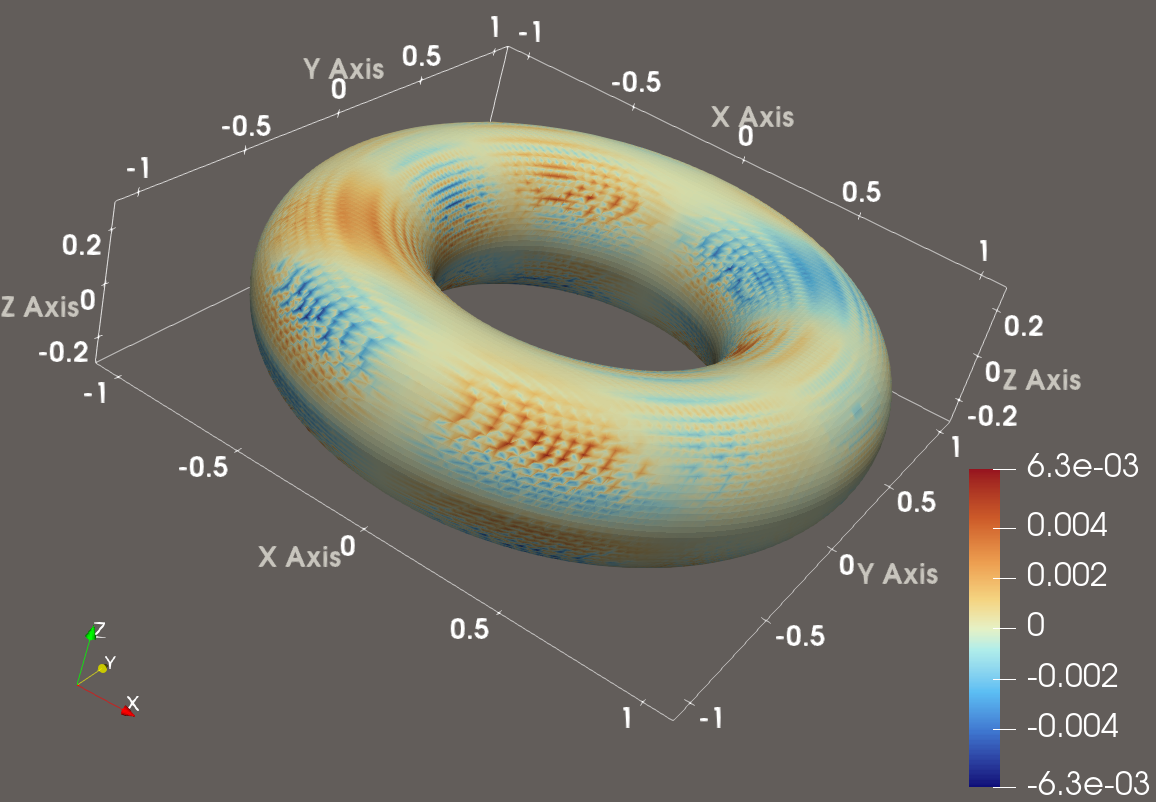}}
    \hfill
    \subfloat[relative residuals]{\includegraphics[width=0.32\textwidth, trim=2cm 7cm 2cm 7cm]{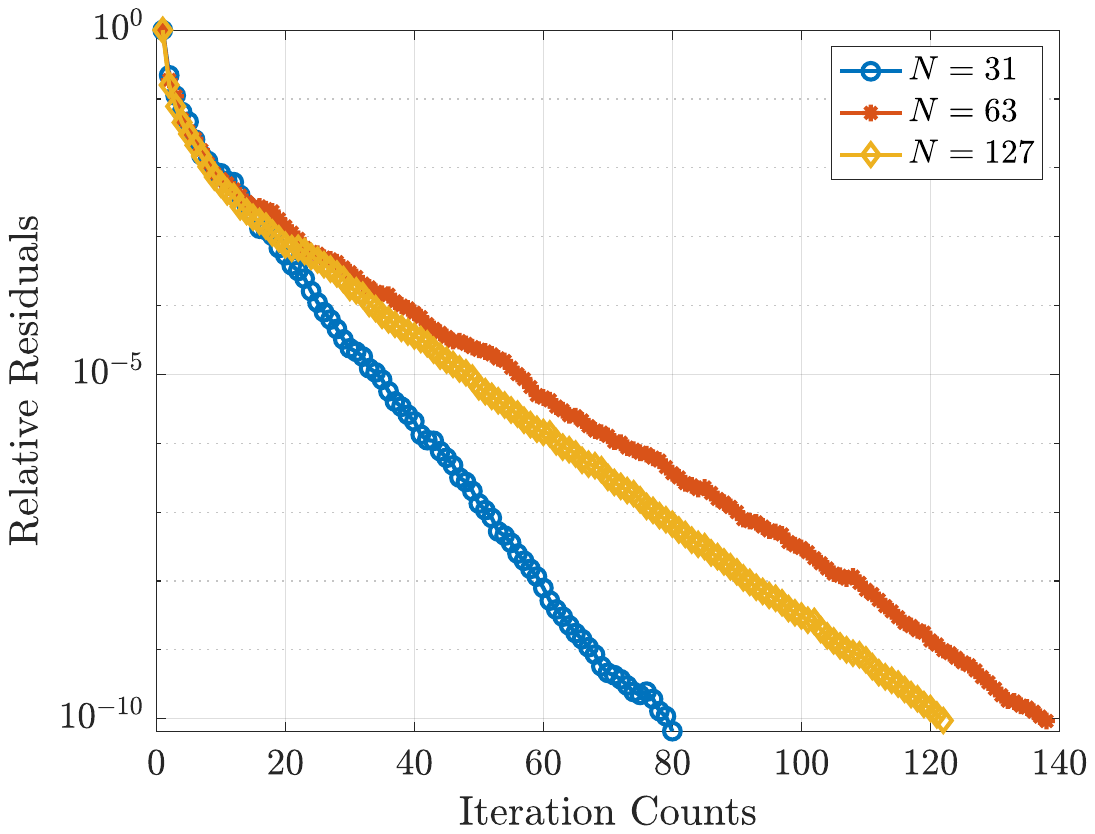}}
    \caption{Errors for the three-dimensional torus on the $N=127$ grid and relative residuals of the iterative solver.}
    \label{fig:torus}
\end{figure}

\Cref{fig:torus} presents the bulk and surface errors on the $127^3$ interior grid and the relative residuals of the \texttt{pcg} solves for all three levels.

\section{Conclusion}\label{sec:conclusion}

We introduced a penalty-free CutFEM formulation in which a lattice Green's function extension constrains the active finite element space and a symmetric Galerkin restriction produces the reduced system. The central idea is to obtain cut stabilization by selecting a discrete-harmonic subspace of the standard active finite element space, rather than by adding a stabilization term to the surface bilinear form. The direct and density parametrizations represent the same reduced space, while the single-layer density coordinates also precondition the second-order surface operator. This connection among discrete harmonic extension, embedded finite element spaces, and operator preconditioning suggests a broader design principle: a stable bulk discrete solver can supply both a cut-robust trial space and effective coordinates for its algebraic solution. The two-dimensional analysis establishes optimal surface approximation under the stated assumptions and cut-robust algebraic estimates for the model problem. The numerical results support the curved-interface hypotheses on circular, deformed, and smooth nonconvex geometries. They also demonstrate the method in three dimensions on a torus, where the measured bulk and surface errors attain the expected rates and the reduced condition numbers and iteration counts remain of order $10^2$ over the three tested refinements.

The present two- and three-dimensional implementations are not matrix-free. Both assemble the dense layer blocks, the reduction map $F$, and the reduced matrix $F^{T}KF$ before passing that matrix to \texttt{pcg}. The reported timings and iteration counts therefore demonstrate the formulation and its conditioning, not near-linear computational complexity. If $M=|\gamma_2|$, storing a dense layer block and applying it directly require $O(M^2)$ memory and work, respectively; here $M=O(h^{-1})$ in two dimensions and $M=O(h^{-2})$ in three dimensions. Translation invariance makes accelerated matrix-free layer application possible with FMM or related machinery~\citep{liska2014parallel,liska2016fast}, but neither such an implementation nor its timing crossover against sparse stabilization is demonstrated here.

The present framework requires a Cartesian background grid and a bulk operator with a tractable LGF. The analysis assumes a smooth closed interface and does not cover interfaces with corners. The curved-interface trace estimate \Cref{ass:curved-trace} and the reduced-space approximation property \Cref{ass:reduced-approx} also remain conditional. The numerical diagnostics support these assumptions for the tested cuts and smooth geometries, but do not replace their proofs or tests over broader geometry classes.

Several other directions are worth pursuing in future work. Replacing the five-point LGF with a nine-point or $\mathcal Q_1$ bulk-element LGF would absorb $\gamma_3$ into $\gamma_2$ and remove the extrapolation $J$. A rigorous extension of the density-formulation conditioning to general curved interfaces would require a discrete pseudodifferential calculus for layer operators restricted to a curved active vertex set. As discussed in \Cref{sec:extensions}, screened Poisson and suitable anisotropic scalar stencils are the most direct further bulk operators; biharmonic and Stokes extensions require separate layer-block and reconstruction-stability results despite the availability of their LGFs. Coupled bulk--surface dynamics and moving interfaces are further natural targets.

\section*{Acknowledgements}
Q.\ Xia would like to thank Sara Zahedi and Xianmin Xu for kind and helpful discussions.
During the preparation of this manuscript, the author used OpenAI Codex for language editing and assistance with \LaTeX{} revision. The author reviewed the resulting text and takes full responsibility for the content.

\section*{Funding}
This work was partially supported by the National Natural Science Foundation of China (Grant No.\ 12401546). Q.\ Xia additionally acknowledges funding from Wenzhou-Kean University (Grant Nos.\ ISRG2024003 and KY20250604000452).

\section*{Conflict of interest}
The author declares no conflict of interest.

\section*{Data availability}
No new data were generated or analysed in support of this research.

\bibliographystyle{plainnat}
\bibliography{references}

\end{document}